\documentclass[11pt,reqno]{amsart}
\usepackage{amssymb,amsfonts,amsthm,amscd,stmaryrd,dsfont,esint,upgreek,constants,todonotes}
\usepackage[mathcal]{euscript}
\usepackage[latin1]{inputenc}   
\usepackage[mathcal]{euscript}
\usepackage{graphicx,color}
\numberwithin{equation}{section}
\newcommand{\N}{\mathbb N}
\newcommand{\Z}{\mathbb Z}

\newcommand{\R}{\mathbb R}

\def\P{\mathbb P}

\DeclareMathOperator*{\esssup}{ess\,sup}
\DeclareMathOperator{\dist}{dist}
\newcommand{\Prob}{\mathbb{P}}
\newconstantfamily{C}{symbol=C}
\newconstantfamily{c}{symbol=c}
\newconstantfamily{O}{symbol=\Omega}

\numberwithin{equation}{section}
\newtheorem{thm}{Theorem}[section]
\newtheorem{lem}[thm]{Lemma}

\newtheorem{prop}[thm]{Proposition}
\theoremstyle{definition}

\def\smallnegint{\mathop{\int\mkern-13mu
        \raise.5ex\hbox{${\scriptscriptstyle\diagup}$}}\nolimits}
\def\ds{\displaystyle}

\def\ep{\varepsilon}

\def\ssetminus{\,\raise.4ex\hbox{$\scriptstyle\setminus$}\,}
\def \lg{\langle}
\def \rg{\rangle}
\newcommand{\be}{\begin{equation}}
\newcommand{\ee}{\end{equation}}

\newcommand{\B}{\mathcal{B}}

\renewcommand{\d}{d}
\newcommand{\Rd}{\mathbb{R}^\d}

\renewcommand{\bar}{\overline}
\renewcommand{\tilde}{\widetilde}
\renewcommand{\hat}{\widehat}

\DeclareMathOperator*{\osc}{osc}

\DeclareMathOperator{\conv}{conv}

\begin{document}
\title[Periodic approximations of the effective equations]{Periodic approximations of the ergodic constants 
in the stochastic homogenization of nonlinear second-order (degenerate) equations}\author[Pierre Cardaliaguet and Panagiotis E. Souganidis]
{Pierre Cardaliaguet \and Panagiotis E. Souganidis}
\address{Ceremade, Universit\'e Paris-Dauphine,
Place du Maréchal de Lattre de Tassigny, 75775 Paris cedex 16 - France}
\email{cardaliaguet@ceremade.dauphine.fr }
\address{Department of Mathematics, University of Chicago, Chicago, Illinois 60637, USA}
\email{souganidis@math.uchicago.edu}
\thanks{\hskip-0.149in Souganidis was partially supported by the National Science
Foundation Grant DMS-0901802. Cardaliaguet was partially supported by the ANR (Agence Nationale de la Recherche) project  ANR-12-BS01-0008-01.
}
\dedicatory{Version: \today}

\begin{abstract} We prove that the effective nonlinearities (ergodic constants) obtained in the stochastic homogenization of Hamilton-Jacobi, ``viscous'' Hamilton-Jacobi and nonlinear uniformly elliptic pde are approximated by the analogous quantities  of appropriate ``periodizations'' of the equations.
We also obtain an error estimate, when there is a rate of convergence for the stochastic homogenization.
\end{abstract}

\maketitle      


\section{Introduction} 

\noindent In this note we prove that the effective nonlinearities arising in the stochastic homogenization of Hamilton-Jacobi, ``viscous'' Hamilton-Jacobi and nonlinear uniformly elliptic pde can be approximated almost surely by the effective nonlinearities of appropriately chosen ``periodizations'' of the equations. We also establish an error estimate in settings for which a rate of convergence is   known for the stochastic homogenization.
\vskip,05in

\noindent To facilitate the exposition, state (informally) our results in the introduction and put everything in context,  we begin by recalling the basic stochastic homogenization results for the class of equations we are considering here.
The linear uniformly elliptic problem was settled long ago by Papanicolaou and Varadhan \cite{PV1,PV2} and Kozlov \cite{Ko}, while general variational problems were studied by Dal Maso and Modica \cite{DD,DD1} (see also Zhikov, Kozlov, and   Ole\u{\i}nik \cite{ZKO}). Nonlinear problems  were considered only relatively recently. Souganidis \cite{S91} and Rezakhanlou and Tarver \cite{RT} considered  the stochastic homogenization of convex and coercive Hamilton-Jacobi equations. The homogenization of viscous Hamilton-Jacobi equations with convex and coercive nonlinearities was established by Lions and Souganidis \cite{LS1,LS2} and Kosygina, Rezakhanlou, and Varadhan \cite{KRV06}. These equations in spatio-temporal media were studied by Kosygina and Varadhan \cite{KV} and Schwab \cite{Sch}.  A new proof for the homogenization yielding convergence in probability was found by Lions and Souganidis \cite{LS3}, and the argument was extended to almost sure by Armstrong and Souganidis  \cite{AS} who also  considered unbounded environments satisfying general mixing assumptions. Later Armstrong and Souganidis \cite{ASo3} put forward a new argument based on the so-called metric problem.
The convergence rate for these problems was considered, first in the framework of Hamilton-Jacobi equations by Armstrong, Cardaliaguet and Souganidis \cite{ACS} and later extended to the viscous Hamilton-Jacobi case by Armstrong and Cardaliaguet \cite{AC}, while Matic and Nolen \cite{MN} obtained variance estimates for first order problems.  All the above results assume that the Hamiltonians are convex and coercive.  The only known result for stochastic homogenization of non coercive Hamilton-Jacobi equations was obtained by Cardaliaguet and Souganidis \cite{CaSo} for the so-called $G$-equation (also see Nolen and Novikov~\cite{NolenNovikov} who considered the same in dimension $d=2$ and under additional  structure conditions). The stochastic homogenization of fully nonlinear uniformly elliptic second-order equations was established by Caffarelli, Souganidis and Wang \cite{CSW}  and   Caffarelli and Souganidis \cite{CS}  obtained a rate of convergence in strongly mixing environments. Armstrong and Smart extended in \cite{AS12} the homogenization result of \cite{CSW} to the non uniformly elliptic setting and, recently,  improved the convergence rate in \cite{AS13}. The results of \cite{CSW, CS} were extended to spatio-temporal setting by Lin \cite{Lin}.
\vskip.05in

\noindent The problem considered in this paper---approximation of the homogenized effective quantities by the effective quantities for periodic problems---is a classical one.  Approximation by periodic problems as a way to prove random homogenization was used first in \cite{PV2} for linear uniformly elliptic problems and later in the context of random walks in random environments by, for example, Lawler \cite{L}  and Guo and Zeitouni \cite{GZ}. The approach of \cite{PV2} can be seen as a particular case of the ``principle of periodic localization" of Zhikov, Kozlov, and   Ole\u{\i}nik \cite{ZKO}  for linear, elliptic problems. Bourgeat and Piatnitski \cite{BP} gave the first convergence estimates for this approximation, while Owhadi \cite{Ow} proved the convergence for the effective conductivity. As far as we know the results in this paper are the first for nonlinear problems and provide a complete answer to this very natural question.
\vskip.05in

\noindent  We continue with an informal discussion of our results. 
Since the statements and arguments for the Hamilton-Jacobi and viscous Hamilton-Jacobi equations are similar here as well as for the rest of the paper we combine them under the heading viscous Hamilton-Jacobi-Bellman equations, for short viscous HJB. Hence our presentation consists of two parts, one for viscous HJB and one for fully nonlinear second-order uniformly elliptic problems. We begin with the former and continue with the latter.
\vskip.05in

\noindent {\it Viscous Hamilton-Jacobi-Bellman equations:} We consider viscous HJB equations of the form
\begin{equation} \label{HJqIntro}
-\ep {\rm tr}\left(A(\tfrac x\ep,\omega)D^2u^\ep\right) + H(Du^\ep,\tfrac x\ep, \omega) = 0, 
\end{equation}
with the  possibly degenerate elliptic matrix $A=A(y,\omega)$ and the Hamiltonian $H=H(p,y,\omega)$  stationary ergodic with respect to $\omega$  and, moreover, $H$  convex and coercive in $p$ and $A$ the ``square'' of  a Lipschitz matrix.
Precise assumptions are given in Section  \ref{sec:CVV}.
Under these conditions, Lions and Souganidis \cite{LS2} proved that almost sure homogenization holds. This means that there exists a convex and coercive Hamiltonian $\overline H$, which we call the ergodic constant,  such that the solution $u^\ep=u^\ep(x,\omega)$ of~\eqref{HJqIntro}, subject to appropriate initial and boundary conditions, converge, as $\ep \to 0$, locally uniformly and  almost surely to the solution $\overline u$ of the deterministic equation, with the same initial and boundary conditions, $\overline H(D \overline u)=0.$
\vskip.05in

\noindent A very useful way to identify the effective Hamiltonian $\overline H(p)$ is to consider the auxiliary problem
\begin{equation}\label{approxcorrector} 
\delta v^\delta - {\rm tr}(A(x,\omega)D^2 v^\delta) + H(D v^\delta +p ,x, \omega) = 0 \quad \mbox{in} \ \R^d, 
\end{equation}
which admits a unique stationary solution $v^\delta(\cdot,\omega)$, often refer to as an ``approximate corrector''. It was shown in \cite{LS1} that, for each $p \in \R^d$, $c>0$ and almost surely in $\omega$, as $\delta \to 0$,
\begin{equation}\label{approximate1}
\sup_{y \in B_{c/\delta}} |\delta v^\delta (\cdot, \omega) + \overline H(p)|  \to 0.  
\end{equation}

\noindent If $A$ and $H$ in \eqref{HJqIntro}  are replaced by $L-$periodic in $y$ maps $A_L(\cdot , \omega)$ and $H_L(\cdot,\cdot, \omega)$, the effective Hamiltonian $\overline H_L(\cdot,\omega)$ is, for any $(p, \omega)\in\R^d\times \Omega$, the unique constant $\overline H_L(p,\omega)$ for which the problem 
\be\label{ergoeqIntro}
- {\rm tr}\left(A_L(x,\omega)D^2\chi \right) + H_L(D\chi+p ,x, \omega) = \overline H_L(p,\omega) \quad \mbox{in} \ \R^d, 
\ee
has a continuous,  $L-$periodic solution $\chi$.  In the context of periodic homogenization,  \eqref{ergoeqIntro} and $\chi$ are called respectively the corrector equation and corrector.
Without any periodicity,   for the constant in \eqref{ergoeqIntro} to be unique, it is necessary for $\chi$ to be strictly sublinear at infinity. As it was shown by Lions and Souganidis \cite{LS1}, in general it is not possible  to find such solutions. The nonexistence of correctors  is the main difference between the periodic and the stationary ergodic settings, a fact which leads to several technical difficulties as well as new qualitative behaviors. 
\vskip.05in

\noindent The very natural question is whether it is possible to come up with $A_L$ and $H_L$ such that, as $L \to \infty$, $\overline H_L(\cdot,\omega)$ converges locally uniformly in $p$ and almost surely in $\omega$  to $\overline H$.  For this it is necessary to choose $A_L$ and $H_L$ carefully.
The intuitive idea, and this was done in the linear uniformly elliptic setting  \cite{BP, Ow, ZKO},   is to take $A_L=A$ and $H_L=H$ in $[-L/2, L/2)^d$ and then to extend them periodically in $\R^d$. 
Unfortunately, such natural as well as simple choice cannot work for viscous HJB equations for two reasons. The first one is that  \eqref{ergoeqIntro}, with the appropriate boundary/initial conditions,  does not have a solution unless $A_L$ and $H_L$ are at least continuous, a property that is not satisfied by the simple choice described above.
The second one, which  is more subtle,  is intrinsically related to the convexity and the coercivity of the Hamiltonian.  Indeed it turns out that viscous HJB equations are very sensitive to large values of the Hamiltonian. As a consequence, the $H_L$'s  must be substantially smaller than $H$ at places where $H$ and $H_L$ differ. 
\vskip.05in

\noindent To illustrate the need to come up with suitable periodizations, we discuss the elementary case when $A\equiv 0$ and $H(p,x,\omega)= |p|^2-V(x,\omega)$ with  $V$ stationary,  bounded and uniformly continuous.  This is one of the very few examples for which the homogenized Hamiltonian is explicitly known for some values of $p$. Indeed $\overline H(0)=\inf_{x\in \R^d}V(x,\omega)$, a quantity  which is independent on $\omega$ in view of the stationarity of $V$ and the assumed ergodicity. If $H_L(p,x,\omega)= |p|^2-V_L(x,\omega)$, with $V_L$ is $L-$periodic, then $\overline H_L(0,\omega)= \inf_{x\in \R^d} V_L(x,\omega)$.  It follows that $V_L$ cannot just be any regularized truncation of $V(\cdot, \omega)$, since it must satisfy, in addition, the  condition $\inf_{x\in \R^d} V_L(x,\omega)\to  \inf_{x\in \R^d}V(x,\omega)$ as $L\to\infty$. 
\vskip.05in

\noindent Here we show that it is possible to choose periodic $A_L$ and $H_L$ so that the ergodic constant  $\overline H_L(p, \omega)$ converges, as $L\to+\infty$ to $\overline H(p)$  locally uniformly in $p$ and almost surely in $\omega$. One direction of the convergence is based on the homogenization, while the other relies on the construction of subcorrectors (i.e., subsolutions)  to \eqref{ergoeqIntro} using  approximate correctors for the original system. We also provide an error estimate for the convergence provided a  rate is known for the homogenization, which is the case for ``i.i.d. environments" \cite{ACS, AC}. 
\vskip.05in
\vskip.05in

\noindent {\it Fully nonlinear, uniformly elliptic equations:} We consider fully nonlinear uniformly elliptic equations of the form 
\be\label{eq:unifellIntro}
F\left(D^2u^\ep, \frac{x}{\ep},\omega\right)= 0, 
\ee
Following Caffarelli, Souganidis  and Wang \cite{CSW}, it turns out that there exists a uniformly elliptic $\overline F$, which we call again the ergodic constant of the homogenization,  such that the solution of  \eqref{eq:unifellIntro}---complemented with suitable boundary conditions---converges, as $\ep\to 0$, to the solution of 
$
\overline F\left(D^2u\right)= 0.$
\vskip.05in

\noindent Although technically involved,  this setting is closer to the linear elliptic one.  Indeed we show that the effective equation $\overline F_L$ of any (suitably regularized) uniformly elliptic (with constants independent of $L$) periodization of $F$ converges almost surely  to $\overline F$.
The proof relies on a combination of homogenization and the Alexandroff-Bakelman-Pucci (ABP)  estimate.  We also give an error estimate for the difference 
$|\overline  F_L(P,\omega) -\overline  F(P)|$, again provided we know rates for the stochastic homogenization like the one's established 
 in \cite{CS} and  \cite{AS13}. 

\subsection*{Organization of the paper}  In the remainder of the introduction we introduce the  notations and some of the terminology needed for the rest of the paper, discuss the general random setting and  record the properties of an auxiliary cut-off function we will be using to ensure the regularity of the approximations. The next two sections are devoted to the viscous HJB equations. In section 2 we introduce the basic assumptions and state and prove the approximation result. Section 3 is about the rate of convergence. The last two sections are about the elliptic problem. In section 4 we discuss the assumptions and state and prove the approximation result. The rate of convergence is the topic of the last section of the paper.


\subsection*{Notation and conventions}
The symbols $C$ and $c$ denote positive constants which may vary from line to line and, unless otherwise indicated, depend only on the assumptions for $A$, $H$ and other appropriate parameters. 
We denote the $\d$-dimensional Euclidean space by $\Rd$,  
$\N$ is the set of natural numbers, 
${\mathcal S}^d$ is the space of $d\times d$ real valued symmetric matrices, and $I_d$ is the $d\times d$ identity matrix . For each $y =(y_1,\dots,y_d)\in \Rd$, $|y|$ denotes the Euclidean length of $y$, $|y|_\infty=\max_i|y_i|$ its $l^\infty-$length,  $\|X\|$ is the usual $L^2$-norm of $X \in {\mathcal S}^d$ and $\lg\cdot,\cdot\rg$ is the standard inner product in $\R^d$. If $E \subseteq\Rd$, then $|E|$ is the Lebesgue measure of $E$ and  $\rm Int (E)$,   $\overline E$ and  $\conv E$ are respectively the interior, the closure and the closure of the convex hull of $E$. We abbreviate almost everywhere to a.e.. We use $\sharp({\mathcal K})$ for the number of points of a finite set ${\mathcal K}$.  For $r>0$, we set $B(y,r): = \{ x\in \Rd : |x-y| < r\}$ and $B_r : = B(0,r)$. For each $z \in \R^d$ and $R>0$, $Q_R(z):=z+[-R/2,R/2)^d$ in $\R^d$ and $Q_R := Q_R(0)$. We say that a map is $1-$periodic, if it is periodic in $Q_1$. 
The distance between two subsets $U,V\subseteq \Rd$ is $\dist(U,V) = \inf\{ |x-y|: x\in U, \, y\in V\}$. If $f:E \to \R$ then $\osc_E f : = \sup_E f - \inf_E f$. 
The sets of functions on a set~$U\subseteq\Rd$  which are  Lipschitz,  have Lipschitz continuous derivatives and are smooth functions 
are written respectively as  
~$C^{0,1}(U)$,  
~$C^{1,1}(U)$ and ~$C^{\infty}(U)$. 
The set of $\alpha-$Hölder continuous functions on $\R^d$ is $C^{0,\alpha}$ and ~$\|u\|$ and $[u_L]_{0,\alpha}$ denote respectively the 
usual $\sup$-norm and $\alpha$-Hölder seminorm. When we need to denote  the dependence of these last quantities on a particular domain $U$ we write 
$\|u\|_U$ and $[u_L]_{0,\alpha;U}$.
The Borel ~$\sigma$-field 
on a metric space $M$ is ~$\mathcal{B}(M)$. If $M=\R^d$, then~$\mathcal{B}=\mathcal{B}(\R^d)$. 
Given  a probability space $(\Omega,\mathcal F, \Prob)$, 
we write a.s. or $\Prob$-a.s.  to abbreviate almost surely.

\noindent Throughout the paper, all differential inequalities are taken to hold in the viscosity sense. Readers not familiar with the fundamentals of the theory of viscosity solutions may consult standard references such as~\cite{B, CIL}.

\subsection*{The general probability setting} Let $(\Omega, \mathcal F, \mathds P)$ be a probability space endowed with a group $(\tau_y)_{y\in \Rd}$ of $\mathcal F$-measurable, measure-preserving transformations $\tau_y:\Omega\to \Omega$. That is, we assume that, for every $x,y\in\Rd$ and $A\in \mathcal F$,
\begin{equation}\label{pres}
\Prob[\tau_y(A)] = \Prob[A] \quad \mbox{and} \quad \tau_{x+y} = \tau_{x} \circ \tau_y.
\end{equation}
Unless we discuss error estimates we take  the group $(\tau_y)_{y\in \Rd}$  to be ergodic. That is, we assume that, if for $A\in \mathcal F$,
\begin{equation}\label{ergodic}
\text{ $\tau_y(A)=A$ for all $y\in \R^d$, then either $\Prob[A]=0$ or $\Prob[A]=1$.} 
\end{equation}
A map $f:M \times \R^d \times \Omega \to \R$, with $M$ either $\R^d$ or ${\mathcal S}^d$, which is measurable with respect to $\mathcal B(M) \otimes \mathcal B \otimes \mathcal F$ is called stationary if  for every 
$m\in M, y,z\in\Rd$ and $\omega \in \Omega$,
$$
f(m,y,\tau_z \omega) = f(m,y+z,\omega).
$$
Given a random variable $f:M \times \R^d \times \Omega \to \R$,  for each $E\in \B$,  let $\mathcal G(E)$ be the $\sigma$-field on $\Omega$ generated by the  $f(m,x,\cdot)$ for $x\in E$ and $m\in M$. We say that the environment is  ``i.i.d.'', if there exists $D > 0$ such that, for all $V,W\in \mathcal B$,
\begin{equation}\label{iid}
{\rm if }\; \dist(V,W) \geq D  \ \ \mbox{then} \ \  \mathcal{G}(V) \ \mbox{and} \ \mathcal{G}(W) \ \mbox{are independent.}
\end{equation}
We say that the environment is strongy mixing with rate $\phi: [0,\infty) \to [0,\infty)$, if  $\lim_{r \to \infty} \phi(r) =0$ and 
\begin{equation}\label{mixing}
\text{if $\dist(V,W) \geq r$ \ then} \ \sup_{A \in \mathcal G(V), B \in \mathcal G(W)} | \Prob[A\cap B] -\Prob[A]\Prob[B] | \leq \phi(r). 
\end{equation}
\subsection*{An auxiliary function} To avoid repetition we summarize here the properties of an auxiliary cut-off function we use in all  sections to construct the periodic approximations. We fix $\eta\in (0,1/4)$ and choose a $1-$ periodic smooth  $\zeta: \R^d \to [0,1]$ so that 
\begin{equation}\label{zeta}
\zeta = 0 \ \text{ in} \ Q_{1-2\eta}, \  \zeta=1 \ \text{in } \ Q_1 \backslash Q_{1-\eta},  \ \|D\zeta\| \leq c/\eta \ \text{ and } \ \|D^2\zeta\| \leq c/\eta^2.
\end{equation} 
To simplify the notation we omit the dependence of $\zeta$ on $\eta$.

\section{Approximations for viscous HJB equations}\label{sec:CVV}

\noindent We introduce the hypotheses and we state and prove the approximation result. 

\subsection*{The hypotheses}

The Hamiltonian $H:\Rd\times\Rd\times \Omega\to \R$ is assumed to be measurable with respect to $\mathcal B \otimes \mathcal B \otimes \mathcal F$.  We write $H=H(p,y,\omega)$ and we require that, for every 
$p,y,z\in\Rd$ and $\omega \in \Omega$,
\begin{equation} \label{stnary}
H(p,y,\tau_z \omega) = H(p,y+z,\omega).
\end{equation}
We continue with the structural hypotheses on $H$. We assume that $H$ is, uniformly in $(y,\omega)$, coercive in $p$, that is    
there exists constants $\Cl[C]{C-Const0}>0$ and $\gamma>1$ such that 
\be\label{coercivite}
{\Cr{C-Const0}}^{-1}|p|^\gamma-\Cr{C-Const0} \leq H(p, x, \omega)\leq \Cr{C-Const0} |p|^\gamma+\Cr{C-Const0},
\ee
and, for all $(y,\omega)$,
\be\label{convex}
 \text{ the map $p \to H(p,y,\omega)$ is convex}.
 \ee
The last assumption can be relaxed to level-set convexity at the expense of some technicalities but we are not pursuing this here.
\vskip.05in

\noindent The required regularity of $H$ is that, for all $x,y,p,q\in\R^d$ and  $\omega\in \Omega$,  
\be\label{Hx-Hy}
|H(p,x,\omega)-H(p,y,\omega)|\leq \Cr{C-Const0} (|p|^\gamma+1) |x-y|
\ee
and 
\be\label{Hp-Hq}
|H(p,x,\omega)-H(q,x,\omega)|\leq \Cr{C-Const0}(|p|^{\gamma-1}+|q|^{\gamma-1}+1)|p-q|.
\ee

\noindent Next we discuss the hypotheses on $A:\R^d\times \Omega \to {\mathcal S}^d$. 
We assume that, for each $(y,\omega) \in \R^d \times \Omega$, there exists a $d \times k$ matrix $\Sigma=\Sigma(y,\omega)$ such that 
\begin{equation}\label{a1}
A(y,\omega)= \Sigma(y,\omega)\Sigma^T(y,\omega). 
\end{equation}
The matrix $\Sigma$ is supposed to be  measurable with respect to $ \mathcal B \otimes \mathcal F$ and stationary, that is for any $y,z\in \R^d$ and $\omega\in \Omega$, 
\begin{equation}\label{a2}
\Sigma(y,\tau_z\omega)= \Sigma(y+z,\omega). 
\end{equation}
It is clear that \eqref{a1} and \eqref{a2} yield that $A$ is degenerate elliptic and stationary.
\vskip.05in
 
\noindent We also assume that $\Sigma$ is Lipschitz continuous with respect to the space variable, i.e., for all $x,y\in \R^d$ and  $\omega\in \Omega$, 
\be\label{Sx-Sy}
|\Sigma(x,\omega)-\Sigma(y,\omega)|\leq  \Cr{C-Const0}|x-y|.
\ee
To simplify statements, we write
\be\label{allH}
\eqref{stnary}, \eqref{coercivite}, \eqref{convex}, \eqref{Hx-Hy} \  \text{and} \  \eqref{Hp-Hq} \ \text{ hold},
\ee
and
\be\label{allA}
\eqref{a1}, \eqref{a2}  \ \text{and} \   \eqref{Sx-Sy}  \  \text{hold}.
\ee

\noindent We denote by $\overline H=\overline H(p)$ the averaged Hamiltonian corresponding to the homogenization problem for $H$ and $A$. We recall from the discussion in the introduction that $\overline H(p)$ is the a.s. limit, as $\delta\to 0$, of $-\delta v^\delta(0,\omega)$, where $v^\delta$ is the solution to \eqref{approxcorrector}.  
Note  (see \cite {LS2}) that, in view of \eqref{convex}, $\overline H$ is convex.  Moreover  \eqref{coercivite} yields (again see \cite{LS2})
\be\label{estBarH}
\Cr{C-Const0}^{-1}|p|^\gamma-\Cr{C-Const0} \leq \overline H(p)\leq \Cr{C-Const0}|p|^\gamma+\Cr{C-Const0}\;.
\ee


\subsection*{The periodic approximation}  Let $\zeta$ be smooth cut-off function satisfying \eqref{zeta}.
\vskip.05in

\noindent For $(p,x,\omega)\in \R^d\times Q_L\times \Omega$ we set
$$
A_L(x,\omega)= (1-\zeta(\frac{x}{L}))A(x,\omega)
$$
and
$$
H_L(p,x,\omega)= (1-\zeta(\frac{x}{L}))H(p,x,\omega)+\zeta(\frac{x}{L})H_0(p),
$$
where,  for a constant $\Cr{defH0}>0$ to be defined below,
$$
H_0(p)= {\Cl[C]{defH0}}^{-1}|p|^\gamma-\Cr{defH0},
$$
Then we extend $A_L$ and $H_L$ to $\R^d\times \R^d\times \Omega$ by periodicity, i.e., for all
 $(x,p)\in \R^d$, $\omega \in \Omega$  and $\xi\in \Z^d$,
$$
H_L(p, x+L\xi,\omega)=H_L(p,x,\omega) \; {\rm and }\; A_L(x+L\xi,\omega)= A_L(x,\omega).
$$
To define $\Cr{defH0}$, let us recall (see \cite{LS2}) that \eqref{allH}  yields  a constant $\Cl[C]{defH01}\geq 1$ such that, for any $\omega\in \Omega$, $p\in \R^d$ and $\delta\in (0,1)$,  the solution $v^\delta$ of \eqref{approxcorrector} satisfies $\|Dv^\delta+p\|_\infty\leq \Cr{defH01}(|p|+1)$. Then we choose $\Cr{defH0}$ so large that, for all $p\in \R^d$,  
$$
\Cr{defH0}^{-1}(\Cr{defH01}(|p|+1))^\gamma-\Cr{defH0}
\leq 
{\Cr{C-Const0}}^{-1}|p|^\gamma-\Cr{C-Const0}. 
$$
In view of \eqref{coercivite}, \eqref{estBarH} and the previous discussion on $v^\delta$, we have
\be\label{pptH0}
H_0(Dv^\delta+p)\leq \overline H(p)\qquad {\rm and} \qquad H_0(p)\leq \overline H(p).
\ee
and, in addition, uniformly in $(y,\omega)$, 
\be\label{ppttH0}
H_L(p,y,\omega) \  \text{is coercive in \ $p$ \ with a constant that depends only on \ $\Cr{C-Const0}$.}
\ee
\vskip.05in

%
%
%
\subsection*{The approximation result} Let $\overline H_L=\overline H_L(p, \omega)$ be the averaged Hamiltonian corresponding to the homogenization problem for $H_L$ and $A_L$. We claim that,  as $L\to +\infty$, $\overline H_L$ is a good a.s. approximation of $\overline H$.
\vskip.05in

\begin{thm}\label{th:mainCV} Assume \eqref{ergodic}, \eqref{allH} and \eqref{allA}. There exists a constant  $C>0$ such that,  for all $p$ and a.s.,
$$
\limsup_{L\to+\infty} \overline H_L(p,\omega) \leq \overline H(p) \leq \liminf_{L\to +\infty} \overline H_{L}(p,\omega) + C(|p|^\gamma+1)\eta.
$$ 
\end{thm}

\noindent Note that, in view of the dependence of $\zeta$ on $\eta$,  $\overline H_L(p,\omega)$ also depends on $\eta$ and, hence, it is not possible to let $\eta \to 0$ in the above statement. However, it is  a simple application of the discussion of Section \ref{sec:rate}, that under suitable assumptions on the random media, it is possible to choose $\eta=\eta_n$ and $L_n\to+\infty$ in such a way that 
$$
\lim_{n\to+\infty} \overline H_{L_n}(p)= \overline H(p).
$$

\subsection*{The proof of Theorem \ref{th:mainCV}}
The proof is divided into two parts which are stated as two separate lemmata. The first is the upper bound, which relies on homogenization and only uses the fact that $H=H_L$  in $Q_{L(1-2\eta)}$. The  second is the lower bound. Here the specific construction of $H_L$ plays a key role. 
\vskip.05in 
\begin{lem}\label{lem:upperbound} ({\it The upper bound}) Assume \eqref{ergodic}, \eqref{allH} and \eqref{allA}. There exists $C>0$ that depends only on $\Cr{C-Const0}$ such that, for all $p\in \R^d$ and a.s.,  
$$
\ds \overline H(p) \leq \liminf_{L\to +\infty} \overline H_L(p,\omega) + C(|p|^{\gamma}+1) \eta.$$
\end{lem}

\begin{proof} Choose  $\omega\in \Omega$ for which homogenization holds (recall that this is the case for almost all $\omega$) and fix  $p\in \R^d$.  Let 
$\chi^p_L$ be a corrector for the $L-$periodic problem, i.e., a continuous, $L-$periodic solution of  \eqref{ergoeqIntro}.
Without loss of generality we assume that $\chi^p_L(0,\omega)=0$. Moreover,  since $H_L$ is coercive, there exists (see \cite{LS2}) a constant $\Cr{defH01}$, which depends only on  $\Cr{C-Const0}$, such that $\|D\chi^p_L+p\|_\infty\leq \Cr{defH01}(|p|+1)$. 
\vskip.05in

\noindent Define $\ds \Phi^p_L(x,\omega): = {L}^{-1} \chi^p_L(Lx,\omega)$. It follows that the $\Phi^p_L$'s  are $1-$periodic,  Lipschitz continuous uniformly in $L$, since$\|D\Phi^p_L+p\|_\infty\leq \Cr{defH01}(|p|+1)$,  and uniformly bounded in $\R^d$, since $\Phi^p_L(0,\omega)=0$. Moreover,  $$
-{L}^{-1}{\rm tr}\left(A_L(Lx,\omega)D^2\Phi^p\right) +H_L(D\Phi^p_L+p, Lx, \omega)= \overline H_L(p,\omega) \  {\rm in} \  \R^d\;. 
$$
Since $H_L= H$ and $A_L=A$  in  $Q_{L(1-2\eta)}$, we also have 
$$
-{L}^{-1}{\rm tr}\left(A(Lx,\omega)D^2\Phi^p_L\right) +H(D\Phi^p_L+p, Lx, \omega)=  \overline H_L(p,\omega) \  {\rm in} \ \text{Int}(Q_{1-2\eta})\;. 
$$
Let $L_n\to+\infty$ be such that $\ds H_{L_n}(p,\omega)\to \liminf_{L\to+\infty} H_L(p,\omega)$. The equicontinuity and equiboundedness of the $\Phi^p_L$'s
yield a further subsequence, which for notational simplicity we still denote by $L_n$, such that the $\Phi^p_{L_n}$'s converge uniformly in $\R^d$ to a Lipschitz continuous, $1$-periodic map $\Phi^p:\R^d \to \R$. Note that, by periodicity
\be\label{e.meanzero}
\ds \int_{Q_1} D\Phi^p=0.
\ee 
\noindent Since homogenization holds for the  $\omega$ at hand, by the choice of the subsequence, we have both in the viscosity and a.e. sense that  
\be\label{HDPhi}
\overline H(D\Phi^p+p)=  \liminf_{L\to+\infty} \overline H_L(p,\omega) \  \mbox{\rm  in} \  \text{ Int} (Q_{1-2\eta}). 
\ee
It then follows from \eqref{e.meanzero}, the convexity of $\overline H$ and Jensen's inequality that 
$$\ds \overline H(p) \leq   \ds \int_{Q_1} \overline H(D\Phi^p+p).$$ 
\noindent Using  the bound on $\|D\Phi^p\|$  
together with \eqref{estBarH} and \eqref{HDPhi}, we get
$$
\begin{array}{rl}
\ds  \int_{Q_1} \overline H(D\Phi^p+p) \; \leq &\ds  \int_{Q_{1-2\eta}} \overline H(D\Phi^p+p)+ \Cr{C-Const0}(\|D\Phi^p+p\|_\infty^\gamma+1) \left| Q_1\backslash Q_{1-2\eta}\right|\\
\leq & \ds  (1-2\eta)^d \liminf_{L\to+\infty} \overline H_L(p,\omega)+  C(|p|^{\gamma}+1)\eta,\\
\end{array}
$$
and, after employing \eqref{estBarH} once more, 
$$
\overline H(p) \; \leq \;  \liminf_{L\to+\infty} \overline H_L(p,\omega)+ C(|p|^{\gamma}+1)\eta\;.
$$
\end{proof}

\noindent To state the next lemma recall that, for any $\delta>0$ and $p\in \R^d$, $v^\delta(\cdot, \omega;p)$ solves  \eqref{approxcorrector}. 

\begin{lem}\label{Lem:owerBound} (The lower bound) Assume \eqref{allH} and \eqref{allA}. For any $K>0$, there exists $C>0$  such that, for all $p\in B_K$,  $L=1/\delta \geq 1$  and  $\ep>0$, 
$$
\{\omega\in \Omega : \sup_{y\in Q_{1/\delta}} | \delta v^\delta(y,\omega;p)+\bar H(p)| \leq \ep\}
\subset
\{\omega\in \Omega: \  \bar H_L(p,\omega)- \bar H(p) \leq 
 \frac{C\ep }{\eta }\}.
$$ 
\end{lem}

\begin{proof} 
Fix  $\omega\in \Omega$  such that 
\be\label{supdeltavdelta}
\sup_{y\in Q_{1/\delta}} | \delta v^\delta(y,\omega;p)+\bar H(p)| \leq \ep,
\ee
and let $\xi: \R^d\to\R$ be a smooth $1-$periodic map such that 
$$\xi=1 \ \text{in} \  Q_{1-\eta}, \  \xi=0 \ \text{in  \ $Q_1\backslash Q_{1-\eta/2}$,} \   
\|D\xi\|_\infty\leq C{\eta}^{-1}  \   \text{ and \ $\|D^2\xi\|_\infty\leq C{\eta^{-2}}$.}$$
\vskip.05in

\noindent Recall that $L=1/\delta$,  define 
$$
\Psi_L(x,\omega):= \xi(\frac{x}{L})v^\delta(x,\omega; p)-(1-\xi(\frac{x}{L})) \frac{\bar H(p)}{\delta} \ {\rm in } \  Q_L
$$
and extend $\Psi_L(\cdot,\omega)$ periodically (with period $L$) over $\R^d$. 
\vskip.05in

\noindent The goal is to estimate the quantity
$\ds\ -  {\rm tr}(A_L(x,\omega)D^2\Psi_L) + H_L(D\Psi_L+p, x, \omega).\ $
In what follows  we argue as if $\Psi_L$ were smooth, the computation being actually correct in the viscosity sense. 
\vskip.05in

\noindent Observe that, if $x\in Q_{1/\delta}$,  then
$$
D\Psi_L(x,\omega)= \xi(\frac{x}{L})Dv^\delta(x,\omega;p)+ \frac{1}{L}D\xi(\frac{x}{L})(v^\delta(x,\omega;p) + \frac{\bar H(p)}{\delta}) .
$$
Hence, in view of \eqref{supdeltavdelta}, 
\be\label{estiDPsi}
| D\Psi_L(x,\omega)- \xi(\frac{x}{L})Dv^\delta(x,\omega;p)| \leq \frac{C\ep}{\eta L\delta} = \frac{C\ep }{\eta } .
\ee
Note that $\Psi_L= v^\delta$ in $Q_{L(1-\eta)}$ while $\zeta(\cdot/L)=1$ in $Q_L\backslash Q_{L(1-\eta)}$. 
It then follows from the definition of $A_L$ and $H_L$ that,  when $x\in Q_L$, 
\be\label{ALHLvdelta}
\begin{array}{l}
\ds -  {\rm tr}(A_L(x,\omega)D^2\Psi_L) + H_L(D\Psi_L+p, x, \omega) \\ 
\quad = \;  
\ds (1- \zeta(\frac{x}{L}))\left[ -  {\rm tr}(A(x,\omega)D^2\Psi_L) + H(D\Psi_L+p, x, \omega) \right]  
+ \zeta(\frac{x}{L}) H_0(D\Psi_L+p)\\[2mm]
\quad = \;  
\ds (1- \zeta(\frac{x}{L}))[ -  {\rm tr}(A(x,\omega)D^2v^\delta) + H(Dv^\delta+p, x, \omega)]  
+ \zeta(\frac{x}{L}) H_0(D\Psi_L+p).
\end{array}
\ee
We now estimate each term in the right-hand side of \eqref{ALHLvdelta} separately. For the first,
in view of  \eqref{approxcorrector} and \eqref{supdeltavdelta}, we have 
$$
\ds  -  {\rm tr}(A(x,\omega)D^2v^\delta) + H(Dv^\delta+p, x, \omega)  = -\delta v^\delta(x,\omega) \leq \bar H(p)+\ep,
$$
while for the second  we use \eqref{estiDPsi}, the convexity of $H_0$ and the Lipschitz bound on $v^\delta$ to find
$$
\begin{array}{rl}
\ds H_0(D\Psi_L+p) \; \leq & \ds H_0(\xi(\frac{x}{L})Dv^\delta+p) + \frac{C\ep}{\eta}  \\
\leq & \ds \xi(\frac{x}{L}) H_0(Dv^\delta+p)+ (1-\xi(\frac{x}{L}))H_0(p) + \frac{C\ep}{\eta},
\end{array}
$$
and,  in view of  \eqref{pptH0}, deduce that 
$$
H_0(D\Psi_L+p) \leq \overline H(p)+ \frac{C\ep}{\eta}.
$$
Combining the above estimates we find (recall that $\eta\in (0,1)$) 
\be\label{ineqALHLvdelta}
\ds -  {\rm tr}(A_L(x,\omega)D^2\Psi_L) + H_L(D\Psi_L+p, x, \omega) \; \leq \; 
 \bar H(p) + \frac{C\ep}{\eta}.
\ee
Since $\Psi_L$ is $L-$periodic subsolution for the corrector equation associated to $A_L$ and $H_L$, the classical comparison of viscosity solutions (\cite{B}) yields
$$
\bar H_L(p,\omega)\leq \bar H(p)+
 {C\ep }/{\eta } 
.$$
\end{proof}
\noindent We are now ready to present the 

\begin{proof}[Proof of Theorem \ref{th:mainCV}]
In view of Lemma  \ref{lem:upperbound}, we only have to show that 
\be\label{amqpourth:mainCV}
\ds \limsup_{L\to+\infty} \overline H_L(p,\omega)\leq \overline H(p).
\ee
Fix $p\in \R^d$, let $v^\delta$ be the solution to \eqref{approxcorrector}, and  recall that  the $\delta v^\delta$'s  converge a.s. to $-\bar H(p)$ uniformly in cubes of radius $1/\delta$, that is, a.s. in $\omega\in \Omega$, for any $\ep>0$, there exists $\bar \delta=\delta(\omega)$ such that, if $\delta\in (0,\bar \delta)$, 
$$
\sup_{y\in Q_{1/\delta}} | \delta v^\delta(y,\omega; p)+\bar H(p)| \leq \ep. 
$$
For such an $\omega$,  Lemma \ref{Lem:owerBound} implies that, for $L=1/\delta$,  
$$
\bar H_L(p,\omega)\leq \bar H(p)+
 {C\ep }/{\eta }
.$$
Letting first $L\to+\infty$ and then $\ep\to 0$ yields \eqref{amqpourth:mainCV}. 
\end{proof}

\section{Error estimate for viscous HJB equations}\label{sec:rate}

\noindent Here we show that it is possible  to quantify the convergence of $\overline H_L(\cdot,\omega)$ to $\overline H$. For this we assume that  we have an algebraic rate of convergence for the solution $v^\delta=v^\delta(x,\omega;p)$ of  \eqref{approxcorrector}
towards the ergodic constant $\bar H(p)$, that is we suppose that there exists $\bar a\in (0,1)$ and, for each $K>0$ and $m>0$, a map 
$\delta\to \Cl[c]{ratevdelta}^{K,m}(\delta)$ with $\lim_{\delta\to 0} \Cr{ratevdelta}^{K,m}(\delta)=0$ such that 
\be\label{HypRateHJ}
\Prob[\sup_{(y,p)\in Q_{\delta^{-m}}\times B_K}|\delta v^\delta(y,\cdot; p)+\overline H(p)| >\delta^{\bar a}]
\leq \Cr{ratevdelta}^{K,m}(\delta);
\ee
notice that \eqref{HypRateHJ} implies that the convergence of $\delta v^\delta(y,\cdot; p)$ in balls of radius $\delta^{-m}$ is not slower than $\delta^{\bar a}$. 
\vskip.05in

\noindent A rate of convergence like \eqref{HypRateHJ} is shown to hold for Hamilton-Jacobi equations in \cite{ACS} and for viscous HJB in \cite{AC} under some additional assumptions on $H$ and the environment.
\noindent
The first assumption, which  is about the shape of the level sets of $H$, is that, for every $p,y \in \R^d$ and $\omega \in \Omega$
$$
H(p,y,\omega) \geq H(0,0,\omega)  \  \text{and} \  \esssup_{\omega \in \Omega} H(0,0,\omega)=0.
$$
As explained in \cite{ACS}, from the point of view of control theory, the fact that there is a common $p_0$ for all $\omega$ at which $H(\cdot,0,\omega)$ has a minimum  provides ``some controllability''. No generality is lost by assuming $p_0=0$ and $ \esssup_{\omega \in \Omega} H(0,0,\omega)=0$.
\vskip.05in

\noindent  As far as the environment is concerned,  \eqref{HypRateHJ} is known to hold  for ``i.i.d."  environments and under an additional suitable condition on $H$ at $p=0$ or, more precisely, for $p$'s in the flat spot of $\overline H$.  What is an  ``i.i.d.'' environment was explained in the introduction. Here for each $E\in \B$, $\mathcal G(E)$ is the 
$\sigma$-field on $\Omega$ generated by the random variables $A(x,\cdot)$ and $H(p,x,\cdot)$ for $x\in E$ and $p\in \R^d$. 
It was shown  in \cite{ACS} that the $\delta v^\delta$'s may converge  arbitrarily slow  for $p$'s such that $\overline H(p)=0$. For Hamilton-Jacobi equations, i.e., when $A\equiv0$, a sufficient condition for \eqref{HypRateHJ} is the existence of constants $\theta>0$ and $c>0$ such that 
$$
\Prob\left[ H(0,0,\cdot) > - \lambda \right] \geq c \lambda^\theta.
$$
\noindent It is shown in \cite{AC} that for viscous HJB equations the above condition has to be strengthened in the following way: there exist $\theta>0$ and $c>0$ such that, for   $(p,x,\omega)\in \R^d\times \R^d\times \Omega$  and  $p \in B_1$,
$$
H(p, x, \omega)\geq c |p|^\theta. 
$$
In both cases, the function $\Cr{ratevdelta}^{K,m}$ is of the form
$$
\Cr{ratevdelta}^{K,m}(\delta)= C\delta^{-r}\exp(-C\delta^{-b}),
$$
where $r$, $C$ and $b$ are positive constants depending on $A$, $H$, $K$ and $m$. 

\noindent The periodic approximation $A_L$ and $H_L$ is the same as in the previous section except that now $\eta \to 0$ as $L\to+\infty$. The rate at which $\eta \to 0$ depends on the assumption on the medium. As before the smooth $1-$periodic $\zeta$ satisfies \eqref{zeta}.
\vskip0.05in

\noindent We choose $\eta_L\to0$ and, to simplify the notation, we write 
$$
\zeta_L(x):=\zeta_{\eta_L}(\frac{x}{L}).
$$
Note that $\zeta_L$ is now $L-$periodic. 
\vskip.05in

\noindent The result is:

\begin{thm}\label{thm:Cvrate}  Assume \eqref{allH}, \eqref{allA} and  \eqref{HypRateHJ} and set  $\eta_L= L^{-\frac{\bar a}{4(\bar a+1)}}.$
For any $K>0$, there exists a constant $C>0$ such that, for $L\geq 1$, 
$$
\Prob [ \sup_{|p|\leq K} \left|\bar H\left(p\right)-
 \bar H_L(p,\cdot)\right|>  CL^{-\frac{\bar a}{4(\bar a+1)}}]  \leq 2 \Cr{ratevdelta}^{ C, 2({\bar a}+1)}(L^{-\frac{1}{2(\bar a+1)}}).
 $$
\end{thm}

\noindent The main idea of the proof, which is reminiscent to the approach of Capuzzo Dolcetta and Ishii  \cite{ICD} and  \cite{ACS}, is that $  \left|\bar H\left(p\right)-
 \bar H_L(p,\cdot)\right|$ can be controlled by  $\left|\delta v^\delta(z,\omega; p')+\bar H(p')\right|$. The later one is estimated by the convergence rate assumption \eqref{HypRateHJ}. As in the proof of Theorem \ref{th:mainCV} it is important to obtain estimates for the lower and upper bounds. Since the former   is a straightforward application of Lemma \ref{Lem:owerBound}, here we only present the details for the latter.

\subsection*{Estimate for the upper bound} We state the upper bound in the following proposition.
\begin{prop}\label{propupperbound} For any $ K>0$, there exists $C>0$ such that, for any $L\geq 1$, $\lambda\in (0,1]$ and $\delta>0$, 
$$
\begin{array}{l}
\ds \{ \omega\in \Omega: \ \sup_{p\in B_K} [\bar H\left(p\right)-
 \bar H_L(p,\omega) ]>  \lambda + C (\eta_L+ L^{-\frac14}\delta^{-\frac12})\} \\
\qquad \qquad \ds \subset  \{\omega\in \Omega: \   \inf_{(z,p') \in Q_{L} \times  B_C} [ -\delta v^\delta(z,\omega; p')-\bar H(p')] < -\lambda \}.
\end{array}
$$
\end{prop}

\begin{proof} 
Fix $p \in  B_K$ and let  $\chi_L^p$ be a continuous, $L-$periodic solution of the corrector equation  \eqref{ergoeqIntro}; recall that $\chi_L^p$ 
is Lipschitz continuous with Lipschitz constant $\bar L=\bar L(K)$. 

\noindent Set $\ep=1/L$ and consider $w^\ep(x):= \ep \chi_L(\frac{x}{\ep})$ which solves  
$$
-\ep{\rm tr}\left(A_L(\frac{x}{\ep},\omega)D^2w^\ep\right) +H_L(Dw^\ep+p, \frac{x}{\ep}, \omega)= \overline H_L(p,\omega)\qquad {\rm in} \  \R^d\;. 
$$
Note that $w^\ep$ is $1-$periodic and Lipschitz continuous with constant $\bar L$. Moreover, in view of the definition of $A_L$ and $H_L$, 
\be\label{eqwepdansQ1moins}
-\ep{\rm tr}\left(A(\frac{x}{\ep},\omega)D^2w^\ep\right) +H(Dw^\ep+p, \frac{x}{\ep}, \omega)= \overline H_L(p,\omega) \  {\rm in} \ Q_{1-2\eta_L}\;. 
\ee
Fix  $\gamma>0$ to be chosen later and consider the sup-convolution $w^{\ep,\gamma}$ of $w^\ep$, which is given by 
$$
w^{\ep,\gamma}(x,\omega):= \sup_{y\in \R^d} ( w^\ep(y,\omega)-\frac{1}{2\gamma}|y-x|^2).
$$
Note that $w^{\ep,\gamma}$ is also $1-$periodic and, by the standard properties of the sup-convolution (\cite{B}, \cite{CIL}), 
$\|Dw^{\ep,\gamma}\|_\infty\leq \|Dw^\ep\|_\infty\leq \bar L$.
\vskip.05in

\noindent The main step of the proof is the following lemma which we prove after the end of the ongoing proof.
\begin{lem}\label{Lem:estimate} 
There exists a sufficiently large constant $C>0$ with the property that,  for any $\kappa >0$ and any $\omega\in \Omega$, if
\be\label{hypopoomega}
\inf_{(z,p') \in Q_{L} \times B_{2\bar L}} [ -\delta v^\delta(z,\omega; p')-\bar H(p')] \geq -\lambda,
\ee
then, for a.e. $x\in Q_r$ with $r = 1-2\eta_L-(C+\bar L)( \gamma +\left(\frac{\ep}{\delta\kappa}\right)^{\frac12})$,
\be\label{eqwepgamma}
\overline H\left(Dw^{\ep,\gamma}(x)+p\right)
\leq 
 \overline H_L(p,\omega)+  \lambda +  C(\frac{1}{\gamma}+\kappa) ( \ep + (\frac{\ep}{\delta\kappa})^{\frac12}).
\ee 
\end{lem}
\noindent We complete the ongoing proof. Jensen's inequality yields, after integrating \eqref{eqwepgamma} over $Q_r$, 
 \begin{multline*}
\overline H\left(r^{-d}\int_{Q_r} Dw^{\ep,\gamma}(x)dx+p\right)
\leq r^{-d}\int_{Q_r}\bar H\left( Dw^{\ep,\gamma}(x)dx+p\right)
\leq \\
 \overline H_L(p,\omega)+  \lambda +  C(\frac{1}{\gamma}+\kappa) ( \ep + (\frac{\ep}{\delta\kappa})^{\frac12}).
 \end{multline*}
The $1-$periodicity of 
 $w^{\ep,\gamma}$ yields $\ds  \int_{Q_1} Dw^{\ep,\gamma}=0$, therefore
\begin{multline*}
\left| \frac{1}{r^{d}} \int_{Q_r} Dw^{\ep,\gamma}\right|\leq \frac{1}{r^d} ( | \int_{Q_1}Dw^{\ep,\gamma} |+   \|Dw^{\ep,\gamma}\|_\infty  \left|Q_1\backslash Q_r|\right)\\ \leq C(1-r)= C(\eta_L+ \gamma +(\frac{\ep}{\delta\kappa})^{\frac12}),
\end{multline*}
with the last equality following from the choice of $r$. 
Hence
$$
\overline H\left(p\right)
\leq 
 \overline  H_L(p,\omega)+  \lambda + C ( \eta_L+\gamma+ \frac{\ep}{\gamma}+ \kappa\ep +(\frac{\ep}{\delta\kappa})^{\frac12}(\frac{1}{\gamma}+\kappa+1)).
$$
Choosing $\ds \kappa= (\ep\delta)^{-\frac13}\gamma^{-\frac23}$ 
and $\ds \gamma= \ep^{\frac14}\delta^{-\frac12}$,
and recalling that $\ep=L^{-1}$, we find 
\be\label{barHleqbarHL}
\overline H\left(p\right)
\leq 
 \overline  H_L(p,\omega)+  \lambda + C(\eta_L+ L^{-\frac14}\delta^{-\frac12}).
\ee
The claim now follows.\\
\end{proof}
\vskip.05in

\noindent We present next the
 
\begin{proof}[Proof of Lemma \ref{Lem:estimate}:]  Let $\bar x\in \text{Int} (Q_r)$ be a differentiability point of  $w^{\ep,\gamma}$. Recall that $w^{\ep,\gamma}$ is
Lipschitz continuous with Lipschitz constant $\bar L$ and, hence, a.e. differentiable.
Then there exists a unique $\bar y\in \R^d$ such that 
\be\label{ypointdemax}
y\to w^\ep(y,\omega)-\frac{1}{2\gamma}|y-\bar x|^2\; \mbox{\rm has a maximum at $\bar y$}
\ee
and
\be\label{calculDwepgamma}
Dw^{\ep,\gamma}(\bar x)= \frac{\bar y-\bar x}{\gamma} \  \text{ and }  \  |\bar y-\bar x|\leq \bar L \gamma.
\ee
Recall that $\kappa>0$ is fixed. For $\sigma>0$ small, consider the map $\Phi:\R^d\times\R^d\times \Omega\to \R$ 
$$
\Phi(y,z,\omega) := w^\ep(y,\omega)-\ep v^\delta \left( \frac z\ep , \omega\, ; \frac{\bar y-\bar x}{\gamma}+p\right) -\frac{|y-\bar x|^2}{2\gamma}
-\frac\kappa2 |y-\bar y|^2
-\frac{|y-z|^2}{2\sigma}
$$
and fix a maximum point $(y_\sigma,z_\sigma)$ of $\Phi$.
\vskip.05in

\noindent Note that $w^{\ep,\gamma}$ as well as all the special points chosen above depend on $\omega$. Since this plays no role in what follows, to keep the notation simple we omit this dependence.
\vskip.05in

\noindent Next we derive some estimates on $|y_\sigma-z_\sigma|$ and $|y_\sigma-\bar y|$.  The Lipschitz continuity of  $\ds v^\delta (\cdot , \omega\, ; \frac{\bar y-\bar x}{\gamma}+p)$ yields 
\be\label{estiysigmamoinszsigma}
|y_\sigma-z_\sigma|\leq C\sigma.
\ee 
Using this last observation as well as the fact that  $(y_\sigma,z_\sigma)$ is a maximum point of $\Phi$ and $\|v^\delta\|_\infty\leq C/\delta$, we get
$$
w^\ep(\bar y,\omega)-C\frac{\ep}{\delta} -\frac{|\bar y-\bar x|^2}{2\gamma} \leq \Phi(\bar y, \bar y) \leq \Phi(y_\sigma,z_\sigma) \leq w^\ep(y_\sigma,\omega)+C\frac{\ep}{\delta}-\frac{|y_\sigma-\bar x|^2}{2\gamma}-\frac\kappa2 |y_\sigma-\bar y|^2,
$$
while, in view of \eqref{ypointdemax}, we also have 
$$
w^\ep (y_\sigma, \omega) - \frac{| y_\sigma-\bar x |^2}{2\gamma} \leq  w^\ep (\bar y,\omega)-\frac{| \bar y-\bar x |^2}{2\gamma}.
$$
Putting together the above inequalities yields
\be\label{ysigmamoinsy}
|y_\sigma-\bar y|\leq C\left( \frac{\ep}{\delta\kappa}\right)^{\frac12}. 
\ee
In particular, \eqref{calculDwepgamma}, the choice of $r$  
and the fact that $\bar x\in \text {Int} (Q_r)$  imply that 
$$ \bar y\in \text{ Int} (Q_{1-2\eta_L-C\left(\frac{\ep}{\delta\kappa}\right)^{\frac12}}) \ \text{ and
$y_\sigma\in Q_{1-2\eta_L}$.}$$ 
\vskip.05in

\noindent  At this point, for the convenience of the reader, it is necessary to recall some basic notation and terminology  from the theory of viscosity solutions (see \cite{CIL}).
Given a viscosity upper semicontinuous (resp.  lower semicontinuous) sub-solution (resp. super-solution ) $u$ of $F(D^2u, Du, u, x)=0$ in some open subset $U$ of $ \R^d$, the lower-jet (resp. upper-jet)  $\overline {\mathcal J}^{2,+}u(x)$ (resp. $ \overline {\mathcal J}^{2,-}u(x)$) at some $x \in U$ consists of $(X,p)\in {\mathcal S}^d \times \R^d$ that can be used to evaluate the equation with the appropriate inequality. For example if $u$ is a sub-solution of $F(D^2u, Du, u, x)=0  \ \text{ in} \ U$ and $(X,p) \in \overline {\mathcal J}^{2,+}u(x)$, then $F(X,p,u(x),x)\leq 0$.
\vskip.05in

\noindent Now  we use the maximum principle for semicontinuous functions (see \cite{CIL}). Since $(y_\sigma,z_\sigma)$ is a maximum point of $\Phi$, for any $\eta>0$, there exist $Y_{\sigma,\eta}, Z_{\sigma,\eta} \in {\mathcal S}^d$ 
such that 
$$
(Y_{\sigma,\eta}, \frac{y_\sigma-\bar x}{\gamma}+\frac{y_\sigma-z_\sigma}{\sigma}+\kappa(y_\sigma-\bar y)) \in \overline {\mathcal J}^{2,+}w^\ep(y_\sigma,\omega),  \ \  ( Z_{\sigma,\eta}, \frac{y_\sigma-z_\sigma}{\sigma})\in \overline {\mathcal J}^{2,-}v^\delta(\frac{z_\sigma}{\ep},\omega),
$$
and 
\be\label{matineq}
\left(\begin{array}{cc}
Y_\sigma & 0\\
0 & \frac{1}{\ep} Z_{\sigma,\eta} \end{array}
\right)
\leq M_{\sigma,\eta}
+ \eta M_{\sigma,\eta}^2,
\ee
where
\be\label{defZsigma}
M_{\sigma,\eta}= \left(\begin{array}{cc} \frac{1}{\sigma} I_d & -\frac{1}{\sigma} I_d \\ -\frac{1}{\sigma} I_d & \frac{1}{\sigma} I_d\end{array}\right)
+ 
\left(\begin{array}{cc} (\frac{1}{\gamma} +\kappa)I_d & 0 \\ 0 & 0\end{array}\right).
\ee
Evaluating  the equations 
for $w^\ep$ and $v^\delta$ at $y_\sigma \in Q_{1-2\eta_L}$ and $z_\sigma \in \R^d$ respectively  we find
\begin{equation}\label{eqpourwep}
-\ep {\rm tr}(A(\frac{y_\sigma}{\ep},\omega)Y_{\sigma,\eta})
+H( \frac{y_\sigma-\bar x}{\gamma}+\frac{y_\sigma-z_\sigma}{\sigma}+\kappa(y_\sigma-\bar y)+p, \frac{y_\sigma}{\ep}, \omega)\leq \overline H_L(p,\omega)
\end{equation}
and 
\be\label{eqpourvdelta}
\delta v^\delta(\frac{z_\sigma}{\ep},\omega;\frac{\bar y-\bar x}{\gamma}+p)
-{\rm tr}\left(A(\frac{z_\sigma}{\ep},\omega\right)Z_{\sigma,\eta})+
H(\frac{y_\sigma-z_\sigma}{\sigma}+\frac{\bar y-\bar x}{\gamma}+p,\frac{z_\sigma}{\ep}, \omega) \geq 0.
\ee
Multiplying  \eqref{matineq} by the positive matrix
$$
\left(\begin{array}{c}
\Sigma\left(\frac{y_\sigma}{\ep},\omega\right) \\
\Sigma\left(\frac{z_\sigma}{\ep},\omega\right)\end{array}\right)
\left(\begin{array}{c}
\Sigma\left(\frac{y_\sigma}{\ep},\omega\right) \\
\Sigma\left(\frac{z_\sigma}{\ep},\omega\right)\end{array}\right)^T,
$$
and taking the trace,  in view of \eqref{defZsigma}, we obtain
\begin{multline*}
 {\rm tr}(A(\frac{y_\sigma}{\ep},\omega)Y_{\sigma,\eta})
 - \frac{1}{\ep} {\rm tr}(A(\frac{z_\sigma}{\ep},\omega)Z_{\sigma,\eta}) \\
 \leq \frac{1}{\sigma}\|\Sigma(\frac{y_\sigma}{\ep},\omega)-\Sigma (\frac{z_\sigma}{\ep},\omega) \|^2
 +(\frac{1}{\gamma}+\kappa) \|\Sigma\left(\frac{y_\sigma}{\ep},\omega\right)\|^2\\
 +\eta {\rm tr}\left (M_{\sigma,\eta}^2\left(\begin{array}{c}
\Sigma\left(\frac{y_\sigma}{\ep},\omega\right) \\
\Sigma\left(\frac{z_\sigma}{\ep},\omega\right)\end{array}\right)
\left(\begin{array}{c}
\Sigma\left(\frac{y_\sigma}{\ep},\omega\right) \\
\Sigma\left(\frac{z_\sigma}{\ep},\omega\right)\end{array}\right)^T
\right).
 \end{multline*}
Recalling that $\Sigma$ satisfies \eqref{Sx-Sy}  and using \eqref{estiysigmamoinszsigma}, we get  
 \begin{multline}\label{ineqpourtraces}
 {\rm tr}(A(\frac{y_\sigma}{\ep},\omega)X_{\sigma,\eta})
 - \frac{1}{\ep} {\rm tr}(A(\frac{z_\sigma}{\ep},\omega)Y_{\sigma,\eta}) \\
 \leq \frac{C}{\sigma\ep^2}|y_\sigma-z_\sigma|^2+C(\frac{1}{\gamma}+\kappa)
 + \eta C(\sigma)
 \leq C(\frac{\sigma}{\ep^2}+\frac{1}{\gamma}+\kappa)
 + \eta C(\sigma).
 \end{multline}
 \noindent  Note that the  $C(\sigma)$ actually depends also on all the other parameters of the problem but, and this is important,  is independent of $\eta$. 
\vskip0.05in

\noindent Next we use that $H$ satisfies \eqref{Hx-Hy}, \eqref{Hp-Hq}.  From  \eqref{estiysigmamoinszsigma} and \eqref{ysigmamoinsy} it follows 
\begin{multline}\label{estiHH}
H( \frac{y_\sigma-\bar x}{\gamma}+\frac{y_\sigma-z_\sigma}{\sigma}+\kappa(y_\sigma-\bar y)+p, \frac{y_\sigma}{\ep}, \omega)
-H(\frac{y_\sigma-z_\sigma}{\sigma}+\frac{\bar y-\bar x}{\gamma}+p,\frac{z_\sigma}{\ep}, \omega)
\geq \\
-C(  \frac{|y_\sigma-\bar y|}{\gamma} +\kappa|y_\sigma-\bar y|+ \frac{|y_\sigma-z_\sigma|}{\ep}) \geq 
-C((\frac{\ep}{\delta\kappa})^{\frac12}(\frac{1}{\gamma}+\kappa)+\frac{\sigma}{\ep}).
 \end{multline}
\vskip.05in

\noindent We estimate the  
difference between \eqref{eqpourwep} and \eqref{eqpourvdelta}, using  \eqref{ineqpourtraces} and \eqref{estiHH}, to find 
$$
-\delta v^\delta\left(\frac{z_\sigma}{\ep},\omega;\frac{\bar y-\bar x}{\gamma}+p\right) \leq 
 \overline H_L(p,\omega)+ C( \frac{\sigma}{\ep}+\frac{\ep}{\gamma}+ \kappa\ep +(\frac{\ep}{\delta\kappa})^{\frac12}(\frac{1}{\gamma}+\kappa)+\frac{\sigma}{\ep})+ \eta \ep C(\sigma).
$$
Finally the choice of $\omega$ (recall  \eqref{hypopoomega}) implies
$$
\overline H(\frac{\bar y-\bar x}{\gamma}+p)
\leq 
 \overline H_L(p,\omega)+ \lambda + C( \frac{\sigma}{\ep}+\frac{\ep}{\gamma}+ \kappa\ep +(\frac{\ep}{\delta\kappa})^{\frac12}(\frac{1}{\gamma}+\kappa)+\frac{\sigma}{\ep})+ \eta \ep C(\sigma).
$$
We now let $\eta\to0$ and then $\sigma \to 0$ to obtain
$$
\overline  H (\frac{\bar y-\bar x}{\gamma}+p)
\leq 
 \bar H_L(p,\omega)+  \lambda + C( \frac{\ep}{\gamma}+ \kappa\ep +(\frac{\ep}{\delta\kappa})^{\frac12}(\frac{1}{\gamma}+\kappa)).
$$
Recalling \eqref{calculDwepgamma}, we may now conclude  \eqref{eqwepgamma} holds. 
\end{proof}

\subsection*{The full estimate}
Combining  Proposition \ref{propupperbound} and Lemma \ref{Lem:owerBound} yields  the full estimate. 

\begin{proof}[Proof of Theorem \ref{thm:Cvrate}] Recall   
that $\eta_L=L^{-\frac{\bar a}{4(\bar a+1)}}$ and choose $\delta= L^{-\frac{1}{2(\bar a+1)}}$ and $\lambda = \delta^{\bar a}=L^{-\frac{\bar a}{2(\bar a+1)}}$. Then Proposition \ref{propupperbound} implies that, for $L\geq 1$,  
$$
\begin{array}{l}
\ds \Prob[ \sup_{p \in B_K} (\overline  H\left(p\right)-
 \overline H_L(p,\cdot))>  CL^{-\frac{\bar a}{4(\bar a+1)}}] \\
\qquad \qquad \ds \leq \Prob[ \inf_{(z,p')\in Q_{L} \times B_C} ( -\delta v^\delta(z,\cdot; p')-\overline H(p')) < -\delta^{\bar a}].
\end{array}
$$
Then  \eqref{HypRateHJ} gives 
$$
\Prob[ \sup_{p\in B_K} (\overline H(p)-
 \overline H_L(p,\cdot))>  CL^{-\frac{\bar a}{4(\bar a+1)}}]  \leq \Cr{ratevdelta}^{ C, 2({\bar a}+1)}(L^{-\frac{1}{2(\bar a+1)}}).
 $$
Similarly Lemma \ref{Lem:owerBound} implies, for $L\geq 1$, that 
 $$
\begin{array}{l}
\ds \{ \omega\in \Omega :  \sup_{p \in B_K} (\overline H_L(p,\omega)- \overline H(p))>   \frac{C\lambda }{\eta_L } 
\} \\
\qquad \qquad \ds \subset \{ \omega\in \Omega:   \sup_{(y,p) \in Q_{L}\times B_K} | \delta v^\delta (y,\omega\, ; p) + \overline H(p)| >\lambda \}.
\end{array}
$$
Since $\lambda/\eta_L = L^{-\frac{\bar a}{4(\bar a+1)}}$, using \eqref{HypRateHJ} we get
$$
\begin{array}{l}
\ds \Prob[  \sup_{p \in B_K} (\overline H_L(p,\cdot)- \overline H(p))> CL^{-\frac{\bar a}{4(\bar a+1)}}] \\
\qquad \qquad \ds \leq\Prob[   \sup_{(y,p) \in Q_{L}\times B_K} | \delta v^\delta (y,\cdot\, ; p) + \overline H(p)| >\delta^{\bar a}]\leq  \Cr{ratevdelta}^{ C, 2({\bar a}+1)}(L^{-\frac{1}{2(\bar a+1)}}).
\end{array}
$$
Combining both estimates gives the result. 
\end{proof}

\section{ Approximations fully nonlinear uniformly elliptic equations}\label{sec:unifell}

\noindent We introduce the hypotheses and state and prove the approximation result.
We remark that our arguments extend to nonlinear elliptic equations which include gradient dependence at the expense of some additional technicalities.

\subsection*{The hypotheses} The 
map $F:{\mathcal  S}^d \times \R^d \times \Omega\to \R$ is $ \mathcal B(\mathcal S^d)\otimes  \mathcal B\otimes \mathcal F$ measurable and stationary, that is,  for all $X\in {\mathcal S}^d$, $x,y\in \R^d$ and $\omega\in \Omega$, 
\be\label{Fstationary}
F(X, y,\tau_z \omega)= F(X, y+z,\omega).
\ee
We continue with the structural hypotheses on $F$. We assume that it is uniformly elliptic uniformly in $\omega$, that is there exist constants 
$0<\overline \lambda<\overline \Lambda$ such that, for all $X,Y \in {\mathcal S}^d$ with $Y\geq 0$, $x \in \R^d$ and  $\omega \in \Omega$,
\be\label{Felliptic}
-\overline \Lambda \|Y\| \leq F(X+Y,x,\omega)-F(X,x,\omega)\leq -\overline \lambda \|Y\|,
\ee
\vskip 0.05in
\noindent and bounded, that is there exists 
$\bar C>0$ such that 
\be\label{Fbounded}
\sup_{\omega\in \Omega} |F(0,0,\omega)| \leq \bar C.
\ee
\vskip.05in

\noindent Note that, in view of \eqref{Fstationary} and \eqref{Felliptic}, for each $R>0$, there exists $C=C(R,\bar \lambda, \bar \Lambda, \bar C)>0$ such that
\be\label{Fbbounded}
\sup_{\|X\|\leq R,\ y\in \R^d,\ \omega \in \Omega} |F(X,y,\omega)| \leq C. 
\ee
\noindent The required regularity on $F$ is that there exists $\rho:[0,\infty) \to [0,\infty)$ such that $\lim_{r\to 0} \rho(r) =0$ and, for all $x,y\in \R^d$, $\sigma>0$, $P,X,Y\in {\mathcal S}^d$ such that 
$$
\left(\begin{array}{cc} X& 0\\0& -Y\end{array}\right) \leq \frac{3}{\sigma}\left(\begin{array}{cc} I_d& -I_d\\-I_d& I_d\end{array}\right),
$$
\be\label{Fcontinuous}
F(X+P,x,\omega)-F(Y+P,y,\omega)\leq \rho (\frac{|x-y|^2}{\sigma}+|x-y|).
\ee
To simplify the statements, we write
\be\label{allF}
\eqref{Fstationary}, \eqref{Felliptic}, \eqref{Fbounded}  \ \text{ and} \ \eqref{Fcontinuous} \  \text{hold}.
\ee
\vskip.05in

\subsection*{The periodic approximation} 
Let $\zeta_\eta:\R^d\to[0,1]$ be a smooth, $1-$ periodic  satisfying \eqref{zeta},
choose $\eta_L\to0$ and, to simplify the notation, write 
$
\zeta_L(x):=\zeta_{\eta_L}(\frac{x}{L}).
$
\vskip.05in

\noindent For $(X,x,\omega)\in \mathcal S^d\times Q_L\times \Omega$ we set
$$
F_L(X,x,\omega)= (1-\zeta_L(x)) F(X,x,\omega)+ \zeta_L(x) F_0(X),
$$
where $F_0\in C({\mathcal S}^d)$ is space independent and uniformly elliptic map with the same ellipticity constants as $F$.  Then we extend $F_L$ to be an  $L-$ periodic  map in $x$, i.e., for all $(X,x,\omega) \in {\mathcal S }^d \times \R^d \times \Omega $  and all $ \xi\in \Z^d,$
$$
F_L(X,x+L\xi,\omega)= F_L(X,x,\omega). 
$$
Note that, in view of the choice of $F_0$, 
\be\label{FL}
F_L \ \text{satisfies} \  \eqref{allF}.
\ee

\subsection*{The approximation result} 
Let $\overline F=\overline F(\cdot)$ and $\overline F_L=\overline F_L(\cdot, \omega)$ be the averaged nonlinearities (ergodic constants) that  correspond to the homogenization problem for $F$ and $F_L(\cdot,\omega)$. We claim that, as $L\to \infty$,  $\overline F_L(\cdot,\omega)$ is an a.s. good approximation of $\overline F$. 

\begin{thm}\label{thm:cvunifell} Assume  \eqref{ergodic} and \eqref{allF}. For any  $P\in{\mathcal S}(\R^d)$ and a.s. in $\omega$,  
\be\label{IneqAProuver}
\lim_{L\to+\infty}  \overline F_L(P, \omega) = \overline F(P).
\ee
 \end{thm}
 
\begin{proof} Fix $P\in {\mathcal S}^d$ with $\|P\|\leq K$, $\omega \in \Omega$ for which the homogenization holds and let $\chi_L$ be a $L-$periodic corrector for $F_L(P,\omega)$, that is a continuous solution to  
$$
F_L(D^2\chi_L+P,x, \omega)= \overline F_L(P, \omega) \ {\rm in } \  \R^d. 
$$
Without loss of generality we assume that $\chi_L(0)=0$. 

\noindent The rescaled function $v_L(x):= L^{-2} \chi_L(Lx)$ is $1-$periodic and solves 
$$
F_L(D^2v_L+P,Lx, \omega)= \overline F_L(P, \omega) \ {\rm in } \ \R^d.
$$
Lemma \ref{lem.pillage} below (its proof is presented after the end of the ongoing one) implies the existence of an  $\alpha \in (0,1)$ and $C>0$ 
depending only $\bar \lambda, \bar \Lambda$, $d$ and $\overline C$ in \eqref{Fbounded} so that 
\be\label{boundsvL}
{\rm osc}(v_L)\leq C \ {\rm and } \  [v_L]_{0,\alpha} \leq C.
\ee
Note that, by the definition of $F_L(P,y,\omega)$, 
$$
F(D^2v_L+P,Lx, \omega)= \overline F_L(P) \  {\rm in } \  Q_{1-2\eta_L},
$$
while, since $F_L$ is uniformly elliptic, 
$$
-M^+(D^2v_L+P)\leq F(P,Lx,\omega)-\overline F_L(P) \ \text{ in } \R^d 
$$
and
$$ 
-M^-(D^2v_L+P)\geq F(P,Lx,\omega)-\overline F_L(P) \ \text{ in } \R^d , 
$$ 
where $M^+$ and $M^-$ are the classical Pucci extremal operators associated with the
uniform ellipticity constants $\bar \lambda/d$ and $\bar \Lambda$ ---see  \cite{CC} for the exact definitions.
\vskip0.05in

\noindent Fix a sequence $(L_n)_{n\in \N} $ such that $\overline F_{L_n}(P, \omega)\to \limsup_{L\to +\infty} \overline F_L(P,\omega)$. 
Using \eqref{boundsvL} we find a further subsequence (still denoted in the same way) and an $1-$periodic $v\in C^{0,\alpha}(\R^d)$ such that the $v_{L_n}$'s converge uniformly to $v$. In view of the results of \cite{CSW} about  stochastic homogenization, $v$ solves 
\be\label{homo1}
\overline F(D^2v+P) = \limsup_{L\to +\infty} \overline F_L(P,\omega) \ {\rm in } \ {\rm Int}(Q_{1}),
\ee
while the stability of solutions also gives, for $C_P= \sup_{x \in \R^d}  |F(P,x,\omega)|+\limsup_{L\to+\infty}|\overline F_L(P)|$, 
$$
-M^+(D^2 v)\leq C_P \ \text{ and } \  -M^-(D^2v)\geq -C_P \  {\rm in } \ R^d.
$$ 
Define  ${B} := \cup_{z\in \Z^d} \partial Q_1(z)$. In view of \eqref{homo1}, the periodicity of $v$ implies that  
\be\label{ineqinB}
\overline F(D^2v+P) =  \limsup_{L\to +\infty} \overline F_L(P,\omega) \qquad {\rm in }\; \R^d\backslash {B}.
\ee
\vskip.05in

\noindent Let $\bar x$ be a minimum point of $v$ and, for $\sigma\in (0,1)$ fixed, consider the map $w(x):= v(x)+\sigma |x-\bar x|^2$ and and its convex hull $\Gamma(w)$. Since $w$ is a subsolution to 
$$
-M^+(D^2w) = -M^+(D^2 v + 2\sigma I_d) \leq C'_P := C_P + 2 \sigma \Lambda, 
$$
it follows (see  \cite{CC})  that $\Gamma(w)$ is $C^{1,1}$ with $\|D^2 \Gamma(w) \|_\infty \leq C$. 
\vskip0.05in

\noindent Let $E$ be the contact set between $w$ and its convex envelope, i.e., 
$$
E:=\{x \in \text{ Int}(B_{1/4}): \ w(x)=\Gamma(w)(x)\}.
$$
Note, and this is a step in the proof of the ABP-estimate (see \cite{CC}) that, if $p\in B_{\sigma/4}$, then there exists  $x\in E$ such that $D\Gamma(w)(x)=p$. Indeed, if $y\notin B_{1/4}(\bar x)$ and $p\in B_{\sigma/4}$, then 
$$
\begin{array}{rl}
\ds w(y)-\lg p,y\rg  \; \geq & \ds v(y)+\sigma |y-\bar x|^2-\lg p,\bar x\rg-| p||y-\bar x|
\\ [2mm]  
> &  v(\bar x)-\lg p,\bar x \rg+ |y-\bar x|(\sigma/4 - |p|) \;
\geq \; \ds  
w(\bar x)-\lg p,\bar x\rg .
\end{array}
$$
Hence any maximum point $x$ of $w-\lg p,\cdot\rg$ must belong to $B_{1/4}(\bar x)$. Then, since  
$w(y)\geq w(x)+\lg p,y-x\rg$ for any $y\in \R^d$,  it follows  that 
$w(x)=\Gamma(w)(x)$ and $D\Gamma(w)(x)=p$. As a consequence we have
$$
|B_{\sigma/4}|\leq |D\Gamma(w)(E)| \leq \int_{E} \det(D^2\Gamma(w))\leq C |E|.
$$
Since  ${B}$ has zero measure, the above estimate shows that there exists $x \in E \backslash B$ such that $w(x)=\Gamma(w)(x)$. Then, for any $y\in \R^d$,  
$$
v(y)\geq \Gamma(w)(y)-\sigma |y-\bar x|^2 \geq w(x) +\lg D\Gamma(w)(x),y-x\rg -\sigma |y-\bar x|^2,
$$
with an equality at $y=x$. 
\vskip0.05in

\noindent Using  
$
\phi(y):= w(x) +\lg D\Gamma(w)(x),y-x\rg-\sigma |y-\bar x|^2
$ as a test function in 
\eqref{ineqinB} and the fact that $x\notin {B}$,  we get
$$
\overline F(-2\sigma I_d+P) \geq \limsup_{L\to +\infty} \overline F_L(P,\omega) . 
$$
Letting $\sigma\to 0$ gives one side of the equality \eqref{IneqAProuver}. 
The proof of the reverse one follows in a symmetrical way.  
\end{proof} 

\noindent To complete the proof, it remains to explain \eqref{boundsvL}. 
For this we note that $v_L$ is $1-$periodic and, in view of \eqref{Fbbounded}, belongs to the class $S^*(\overline \lambda/d, \overline \Lambda, C_0)$, where $C_0= \sup_x |F(P,x,\omega)|$ (see  \cite{CC} for the definition of $S^*(\overline \lambda/d, \overline \Lambda, C_0)$). Then  \eqref{boundsvL} is a consequence of the classical Krylov-Safonov result about the continuity of solutions of uniformly elliptic pde. Since we do not know an exact reference for \eqref{boundsvL}, we present below its proof.

\begin{lem}\label{lem.pillage} Let $0<\lambda<\Lambda$ and $C_0>0$ be constants. There exist $C=C(d, \lambda, \Lambda, d, C_0)>0$ and $\alpha=\alpha (d, \lambda, \Lambda, d, C_0) \in (0,1]$ such that, any $1-$periodic $u \in S^*(\lambda, \Lambda, C_0)$ satisfies
${\rm osc}(u)\leq C$ and $[u]_{C^{0,\alpha}}\leq C$. 
\end{lem}

\begin{proof} Without loss of generality, we assume that $u(0)=0$. For $M\geq 1$, let $u_M(x):=M^{-2}u(Mx)$. Note that $u_M$ is $M^{-1}-$periodic and still belongs to  $S^*(\lambda, \Lambda, d, C_0)$. The Krylov-Safonov result   yields $C=C(\lambda, \Lambda, d, C_0)>0$ and $\alpha=\alpha (d, \lambda, \Lambda, d, C_0) \in (0,1]$ with
$$
[u_M]_{0,\alpha ;Q_1} \leq C( \|u_M\|_{Q_2}
+1).
$$
It follows from the $1-$periodicity of $u$ and $u(0)=0$ that 
$$
[u_M]_{0,\alpha ;Q_1}\geq M^{\alpha-2}[u]_{0,\alpha} \  {\rm and} \  \|u_M\|_{Q_2}= M^{-2}\|u\| \leq M^{-2}d^{\frac12}[u]_{0,\alpha}.$$ 
\noindent Hence,
$$[u]_{0,\alpha} \leq CM^{2-\alpha}( M^{-2}d^{\frac12}[u]_{0,\alpha}+1).
$$
Choosing $M$ so that $CM^{-\alpha}d^{\frac12}=\frac12$ gives a bound on $[u]_{C^{0,\alpha}}$, from which we derive a $\sup$ bound on $u$.  
\end{proof}

\noindent We remark that the proof of Theorem \ref{thm:cvunifell} shows the following fact, which we state as a separate proposition, since it may be of independent interest. 

\begin{prop} Let $\Sigma \subset \R^d$ be a set of zero measure and assume that $u \in C(\R^d)$  
is a viscosity solution of the  uniformly elliptic equation
$F(D^2 u, x)= 0 \  {\rm in } \ \R^d\backslash \Sigma.$
If, in addition, $u \in S^*(\lambda,\Lambda, d, C)$ in $\R^d$, for some $0<\lambda< \Lambda$, then  $u$ is a viscosity solution of $F(D^2u, x)= 0$ in $\R^d$. 
\end{prop}

\section{The convergence rate for nonlinear elliptic equations}\label{sec:unifellerror}

\noindent Here we show that it is possible to quantify the rate of convergence of $\overline F_L(\cdot,\omega)$ to $\overline F$.  As in the viscous HJB problem, we assume that we know a rate for the convergence of the solution to the approximate cell problem to the ergodic constant $\overline F$. 
\vskip.05in

\noindent For the sake of the presentation below and to simplify the argument it is more convenient to consider, for $L\geq 1$ and $P \in {\mathcal S}^d$, the solution  $v^L=v^L(\cdot; P,\omega)$ of 
\be\label{eq.vL}
v^L + F(D^2 v^L+P,Lx,\omega)=0 \ {\rm in } \ \R^d.
\ee
 \noindent Note that $v_L(x): =L^{2}v^L(L^{-1}x)$ solves the auxiliary problem
$$L^{-2}v_L + F(D^2v_L+P, x,\omega)=0 \ \text{ in } \ \R^d.$$
In view of the stochastic homogenization,  it is known that  the $v^L$'s  converge locally uniformly and a.s.  to the unique solution $\overline v= -\overline F(P)$ of 
$$
\overline  v + \overline F(D^2 \bar v+P)=0 \ \text{ in } \ \R^d.
$$
We assume that there exist nonincreasing rate maps $L \to \lambda(L)$ and  $L\to \Cl[c]{elliptic3}(L)$, which tend to $0$ as $L\to+\infty$, such that 
\be\label{HypRateUnifEll}
\P[ \sup_{x\in B_5} | v^L(x,\cdot)+\overline F(P)|> \lambda(L)] \leq  \Cr{elliptic3}(L).
\ee
Recall that such a rate was obtained in \cite{CS} under a strong mixing assumption on the random media---see at the beginning of the paper for the meaning of this. The recent contribution \cite{AS13} shows that for ``i.i.d'' environments the rate is at least algebraic, that is $\lambda (L)= L^{-\bar a}$ for some $\bar a\in (0,1)$. 
\vskip.05in

\noindent In addition to the above assumption on the convergence rate, it also necessary  to enforce the regularity condition \eqref{Fcontinuous} on $F$. Indeed  we assume that there exists a constant $C>0$ such that
\be\label{StructCondF}
\eqref{Fcontinuous} \ \text{ holds with  } \rho(r)=Cr. 
\ee
\vskip.05in

\noindent The periodic approximations of $F$ is exactly the same as in the previous section and the result is:
\vskip.05in

\begin{thm} Assume \eqref{allF}, \eqref{HypRateUnifEll} and \eqref{StructCondF} and  set  $\eta_L(L)= \lambda(L)^{\frac{d}{2d+1}}$.
Then there exists a constants $C>0$ such that, for $L\geq 1$,  
$$
\P[ |-\overline F(P)+  F_L(P, \cdot)| >C\lambda(L)^{\frac{1}{2d+1}}] \leq  \Cr{elliptic3}(L). 
$$
\end{thm}

\begin{proof} For any $L\geq 1$, let $v^L$  be the solution to \eqref{eq.vL}
and $\chi_L$ a $L-$periodic corrector for $F_L(P,\omega)$, i.e., a solution to 
$$
F_L(D^2\chi_L+P,x, \omega)= \overline F_L(P, \omega) \  {\rm in }  \ \R^d. 
$$
Without loss of generality we assume that $\chi_L(0,\omega)=0$. 
Set $w_L(x,\omega):= L^{-2} \chi_L(Lx,\omega)$ and note that $w_L$ is $1-$periodic and solves 
\be\label{eq.vdeltaL}
F_L(D^2 w_L+P,Lx,\omega)=\overline F_L(P, \omega) \ {\rm in } \ \R^d. 
\ee
It also follows from  \eqref{boundsvL} that 
$\|w_L\| \leq C$ and $\left[ w_L\right]_{{0,\alpha}}\leq C,$
and we note that the $v^L$'s  are   bounded in $C^{0,\alpha}$ uniformly with respect to $L$. 
\vskip.05in

\noindent From now on we fix $\lambda>0$ and $\omega$ such that 
\be\label{condomega}
\sup_{x\in B_5} | v^L(x,\omega)+\overline F(P)|\leq \lambda.
\ee
The goal is to show that 
\be\label{ineqcherchee}
\left|\overline F(P)-  F_L(P, \omega)\right| \; \leq \; \ds  \lambda + C( \eta_L^{\frac1d}  +\lambda \eta_L^{-2}). 
\ee
For this, we follow the  proof of the convergence quantifying each step  in an appropriate way. In what follows to simplify the expressions we suppress the dependence on $\omega$ which is fixed throughout the argument. Moreover since the proof is long we organize it in separate subsections.\\
\vskip.05in

\noindent {\it Construction of the minimum points $\bar x_0$ and $\hat x_0$:} Let $\bar x_0$ be a minimum point of $w_L$ in $\overline Q_1$; note that since $w_L$ is $1-$periodic, $\bar x_0$ is actually a minimum point of $w_L(\cdot)$ in $\R^d$. For $a, r, r_0\in(0,1)$ to be chosen below and $\xi\in \R^d$, we consider the map
$$
\Phi^0_\xi(x):= w^L(x)+\frac{a}{2} |x-\bar x_0|^2 +\lg \xi, x-\bar x_0\rg. 
$$
The  claim is that, if $\hat x_0$ is a minimum point of $\Phi^0_\xi$ with  $\xi\in B_{r_0}$ and
\be\label{defr0bis}
2r_0\leq ar, 
\ee
then $\hat x_0\in B_r(\bar x_0)$.
Indeed, by the definition of $\hat x_0$ and $\bar x_0$, 
$$
\Phi^0_\xi(\hat x_0)\leq \Phi^0_\xi(\bar x_0)= w^L(\bar x_0)\leq w^L(\hat x_0),
$$
hence,
$$
\frac{a}{2} | \hat x-\bar x_0|^2 +\lg \xi, \hat x-\bar x_0\rg \leq 0, 
$$
and, in view of \eqref{defr0bis}, $$|\hat x_0-\bar x_0|\leq 2 |\xi|/a\leq 2r_0/a\leq r.$$

\noindent  Next we consider two cases depending on whether $\hat x_0\in Q_1$ or not.
\vskip.05in

\noindent{\it  Case $1$: $\hat x_0\in Q_1$.} Let 
$$E^0:=\{\hat x_0\in Q_1\cap B_r(\bar x_0): \text{ there exists $\xi\in B_{r_0}$ such that $\Phi^0_\xi$ has a minimum at $\hat x_0$}\}.$$
Note that, by the definition of $\hat x_0$, $w_L$ is touched from below at $\hat x_0$ by a parabola of opening $a\in (0,1)$. It then follows from the Harnack inequality that $w_L$ is touched from above at $\hat x_0$ by a parabola of opening $C$. This is a classical fact about uniformly elliptic second-order equations and we refer to \cite{CS1} for more details. 
It follows that $w_L$ is differentiable at $\hat x_0$ and, in view of  the choice of $\hat x_0$ for $\Phi^0_\xi$, 
$
Dw^L(\hat x_0)+a(\hat x_0-\bar x_0) +\xi=0. 
$
\vskip.05in

\noindent Hence $\xi$ is determined from  $\hat x_0$ by the relation  $\xi= \Psi^0(y):= -\left(Dw^L(\hat x_0)+a(\hat x_0-\bar x_0)\right)$. Moreover, in view  the above remark on the parabolas touching $w_L$ from above and below, $\Psi^0$ is Lipschitz continuous on $E^0$ with a Lipschitz constant bounded by $C$. We refer the reader to \cite{CS1} for the details of this argument.\\ 
\vskip.05in

\noindent {\it Case $2$: $\hat x_0\notin Q_1$:} If $\hat x_0\notin Q_1$, then there exists $z\in \Z^d$ such that $\hat x_0\in Q_1(z)$. Since $\hat x_0\in Q_r(\bar x)$ with $\bar x\in Q_1$ and $r<1$, it follows that  $|z|_\infty=1$. Set $\bar x_z:= \bar x_0-z$, $\hat x_z:=\hat x_0-z\in Q_1$ (note that $\hat x_z\in B_r(\bar x_z)$) and 
$$
\Phi^z_\xi(x)= w^L(x)+\frac{a}{2} |x-\bar x_z|^2 +\lg \xi, x-\bar x_z\rg. 
$$
In view of the periodicity of $w_L$, $\hat x_z$ is a minimum point of $\Phi^z_\xi$. 
\vskip.05in

\noindent  Let ${\mathcal Z}:=\{z\in \Z^d\ ; \ |z|_\infty\leq 1\}$. It is clear that  ${\mathcal Z}$ is a finite set and, if $z\in {\mathcal Z}$, either $|z|_\infty=1$ or $z=0$.  
Also set  $E^z$ to be the set of points $\hat x_z\in Q_1\cap B_r(\bar x_z)$ for which there exists $\xi\in B_{r_0}$ such that $\Phi^z_\xi$ has a minimum at $\hat x_z$. 
\vskip.05in

\noindent Arguing as in the previous case, we see that there is a Lipschitz map $\Psi^z$ on $E^z$ (with a uniform Lipschitz constant bounded) such that, if $\hat x_z\in E^z$ and $\xi=\Psi^z(\hat x_z)$, then $\hat x_z$ is a minimum of $\Phi^z_\xi$. \\
\vskip.05in

\noindent {\it The existence of interior minima.} It follows from the previous two steps that, for any $\xi\in B_{r_0}$, there exist $z\in {\mathcal Z}$ and  $\hat x_z\in E^z$ such that $\xi=\Psi^z(\hat x_z)$. Hence, using that the $\Psi^z$'s  are Lipschitz continuous with a uniform Lipschitz constant uniformly for $z\in {\mathcal  Z}$, we find   
$$
| B_{r_0}| = | \cup_{z\in {\mathcal Z}} \Psi^z(E^z)| \leq \sharp({\mathcal Z}) \max_z |\Psi^z(E^z)| \leq C \max_z |E^z|.
$$
Therefore there must exist some  $z\in {\mathcal Z}$, which we fix from now on, such that 
\be\label{Ezgeq}
 |E^z|\geq r_0^d /C.
 \ee
 We now  show that, for a suitable choice of the constants, the sets $E^z$ and $Q_{1-3\eta_L}$ have a nonempty intersection. For this we note that, since $E^z\subset B_r(\bar x_z)\cap Q_1$, the claim holds true as soon as 
 $$
 \left| E^z\right| + \left| B_r(\bar x_z)\cap Q_{1-3\eta_L}\right| >  \left| B_r(\bar x_z)\cap Q_{1}\right| .
 $$
As
 $$
 \left| B_r(\bar x_z)\cap Q_{1}\right|- \left| B_r(\bar x_z)\cap Q_{1-3\eta_L}\right| = \left| B_r(\bar x_z)\cap (Q_1\backslash Q_{1-3\eta_L})\right|\leq C r^{d-1} \eta_L, 
 $$
 provided $r\geq C\eta_L$, we conclude from \eqref{Ezgeq}, that if 
 \be\label{Condr00}
 r_0^d \geq Cr^{d-1} \eta_L, 
\ee
then  $E^z\cap Q_{1-3\eta_L}\neq\emptyset$.\\
\vskip.05in

\noindent {\it The perturbed problem:} From now on we assume that \eqref{defr0bis} and  \eqref{Condr00} hold and, therefore, there exist (fixed) $z\in {\mathcal Z}$, $\xi\in B_{r_0}$ and $\hat x_z\in Q_{1-3\eta_L}\cap B_r(\bar x_z)$ such that $\Phi^z_\xi$ has a minimum at $\hat x_z$. 
For $b,\sigma\in (0,1)$ to be chosen below, set
$$
\begin{array}{rl}
\ds \Phi_\sigma(x,y) : = & \ds  v^L(x)-w_L(y)-\frac{a}{2}|y-\bar x_z|^2-\frac{b}{2}|y-\hat x_z|^2 -\lg \xi, y-\bar x_z\rg -\frac{|x-y|^2}{2\sigma} \\[2mm]
= & \ds v^L(x)-\Phi^z_\xi(y) -\frac{b}{2}|y-\hat x_z|^2  -\frac{|x-y|^2}{2\sigma}
\end{array}
$$
and let $(\tilde x, \tilde y)$ be a maximum point of $\Phi_\sigma$ over $Q_5\times \R^d$. We claim that 
\be\label{estitilde}
|\tilde y-\hat x_z|\leq 2(\lambda b^{-1})^{\frac12}\qquad {\rm and }\qquad |\tilde y-\tilde x|\leq C\sigma^{\frac{1}{2-\alpha}}.
\ee
Indeed, since $\Phi_\sigma(\tilde x,\tilde y)\geq \Phi_\sigma(\hat x_z,\hat x_z)$ and $\Phi^z_\xi(\hat x_z)\leq \Phi^z_\xi(\tilde y)$, in view of \eqref{condomega},
we find  
$$
\begin{array}{rl}
\ds \Phi_\sigma(\tilde x,\tilde y)\; \geq & \ds  v^L(\hat x_z) -\Phi^z_\xi(\hat x_z)\; \geq \overline F(P)-\lambda -\Phi^z_\xi(\tilde y)\\ [2mm]
\geq & v^L(\tilde x)-2\lambda -\Phi^z_\xi(\tilde y),
\end{array}
$$
and, therefore, 
$$
\frac{b}{2}|\tilde y-\hat x_z|^2 +\frac{|\tilde x-\tilde y|^2}{2\sigma}\leq 2\lambda.
$$
This gives the first inequality in \eqref{estitilde}. The maximality  of $\tilde x$ in $\Phi_\sigma(\cdot, \tilde y)$ and the Hölder regularity of $v^L$ gives the second inequality. 
\vskip.05in
 
\noindent If we assume that 
\be\label{condB}
2(\lambda b^{-1})^{\frac12}\leq \eta_L, 
\ee
then, since $\hat x_z\in Q_{1-3\eta_L}$ and \eqref{estitilde} holds,  it follows that $\tilde y$ belongs to $Q_{1-2\eta_L}$. Moreover, for $\sigma$ small enough, we still have  by \eqref{estitilde} that $\tilde x \in Q_2$. In particular,  $(\tilde x,\tilde y)$ is an interior maximum of $\Phi_\sigma$ in $Q_5 \times \R^d$. \\
\vskip.05in

\noindent {\it The maximum principle:} Since $(\tilde x,\tilde y)$ is an interior maximum of $\Phi_\sigma$, the maximum principle already used earlier states that  there exist $X,Y\in {\mathcal S}^d$ and $p_x,p_y\in \R^d$ such that 
$$
(p_x , X)\in \overline {\mathcal J}^{2,+}v^L(\tilde x), \qquad (p_y, Y)\in   \overline {\mathcal J}^{2,-}w_L(\tilde y)
$$
and 
$$
\left(\begin{array}{cc}
X & 0 \\ 0 & -Y-(a+b)I_d\end{array}\right) \leq \frac{3}{\sigma}\left(\begin{array}{cc}
I_d & -I_d \\ -I_d & I_d\end{array}\right).
$$
In view of \eqref{eq.vL} and \eqref{eq.vdeltaL} and since $\tilde y\in Q_{1-2\eta_L}$, which yields that  $F_L(\cdot,\tilde y)=F(\cdot,\tilde y)$, evaluating the equations satisfied by $v^L$ and $w_L$  at $\tilde x$ and $\tilde y$ respectively we find
\be\label{eq.vLappliii}
v^L(\tilde x) + F(X+P,L\tilde x)\leq 0 
\ee
and 
$$
F(Y+P,L\tilde y)\geq \overline F_L(P, \omega).
$$
From the uniform in $x$ and $\omega$ Lipschitz continuity of $F$ with respect to $P$, we get 
$$
F(Y+(a+b)I_d+P,L\tilde y)\geq \overline F_L(P, \omega)- C(a+b).
$$
Using  \eqref{StructCondF} to estimate  the difference between \eqref{eq.vLappliii} and the above inequality, we obtain
$$
v^L(\tilde x) +  \overline F_L(P, \omega) \leq C( \frac{L^2|\tilde x-\tilde y|^2}{\sigma} + L|\tilde x-\tilde y|+ a+b).
$$
Finally we use \eqref{condomega} and \eqref{estitilde} to conclude that
$$
\ds -\overline F(P)+  \overline F_L(P, \omega) \; \leq \; \ds  \lambda+ C( L^2\sigma^{\alpha/(2-\alpha)}+L\sigma^{1/(2-\alpha)}+ a+b).
$$
Letting  $\sigma\to 0$, we get
\be\label{betterestimate2}
\ds -\overline F(P)+  F_L(P, \omega) \; \leq \; \ds  \lambda+ C\left(  a+b\right)
\ee
provided \eqref{condomega}, \eqref{defr0bis}, \eqref{Condr00} and \eqref{condB} hold. \\
\vskip.05in

\noindent {\it The choice of the constants: } In order for  \eqref{condB}, \eqref{defr0bis} and \eqref{Condr00} to hold, we choose respectively $b = 4\lambda \eta_L^{-2},$ $r=1/2$, $r_0= a/4$ 
and $a= C\eta_L^{\frac1d}$. 
We then have 
$$
\ds -\overline F(P)+  F_L(P, \omega) \; \leq \; \ds  \lambda+ C( \eta_L^{\frac1d}  +\lambda \eta_L^{-2}).
$$
Arguing in a similar way, we can check that, under  \eqref{condomega}, we also have 
$$
\ds -\overline F(P)+  F_L(P, \omega) \; \geq \; \ds  -\lambda- C( \eta_L^{\frac1d}  +\lambda \eta_L^{-2}),
$$
which  yields  \eqref{ineqcherchee}. 
\vskip.05in

\noindent Combining \eqref{ineqcherchee} with the assumption \eqref{HypRateUnifEll} on the convergence rate, we find 
 \begin{multline*}
\ds \P[ |-\overline F(P)+  F_L(P, \cdot)| >\lambda(L)+ C( \eta_L^{\frac1d}  +\lambda(L) \eta_L^{-2})] \\
\qquad \qquad  \leq \; \ds \P[ \sup_{x \in B_5} \left| v^L(x,\cdot)+\overline F(P)\right|> \lambda(L)] \leq \; \ds  \Cr{elliptic3}(L).
 \end{multline*}
The  choice of $\eta_L(L)= (\lambda(L))^{\frac{d}{2d+1}}$ gives the claim.
\end{proof}

\newcommand{\noop}[1]{} \def\cprime{$'$} \def\cprime{$'$}


\begin{thebibliography}{10}

\bibitem{ACS} S.~N. Armstrong, P. Cardaliaguet and P.~E. Souganidis.
\newblock Error estimates and convergence rates for the stochastic homogenization of Hamilton-Jacobi equations. 
\newblock Journal of AMS, to appear.
 \newblock arXiv:1206.2601.
 
 \bibitem{AC} S.~N. Armstrong and P. Cardaliaguet.
\newblock In preparation.

\bibitem{AS12}
S.~N. Armstrong and C.~K. Smart. \newblock  Regularity and stochastic homogenization of fully nonlinear equations without uniform ellipticity. \newblock arXiv:1208.4570.

\bibitem{AS13}
S.~N. Armstrong and C.~K. Smart. \newblock  Quantitative stochastic homogenization of elliptic equations in nondivergence form. \newblock
arXiv:1306.5340.



\bibitem{AS}
S.~N. Armstrong and P.~E. Souganidis.
\newblock Stochastic homogenization of {H}amilton-{J}acobi and degenerate
  {B}ellman equations in unbounded environments.
\newblock {\em J. Math. Pures Appl.}, 97 (2012), 460--504.


\bibitem{ASo3}
S.~N. Armstrong and P.~E. Souganidis.
\newblock Stochastic homogenization of level-set convex {H}amilton-{J}acobi
  equations.
\newblock {\em Int. Math. Res. Not.}, \noop{3001}in press.
\newblock Arxiv:1203.6303.

\bibitem{B} G. Barles. \newblock Solutions de viscosite des equations de Hamilton-Jacobi. Volume 17 of Mathematiques and Appplications (Berlin).
Springer Verlag, Paris, 1994.

\bibitem{BP} A. Bourgeat and A. Piatnitski.
\newblock Approximations of effective coefficients
in stochastic homogenization
\newblock  {\em Ann. I. H. Poincare - Probab. Statist.}, 40 (2004), no. 2,  153--165.

\bibitem{CC}  
L.A. Caffarelli and X. Cabre. \newblock  Fully nonlinear elliptic partial differential equations, \newblock  Amer. Math. Soc., 1997.


\bibitem{CS}  L. A. Caffarelli and P.E. Souganidis. \newblock Rates of convergence for the homogenization of
fully nonlinear uniformly elliptic pde in random media. \newblock \textit{Invent. Math.}, 180 (2010), no. 2,  301--360.

\bibitem{CS1} L. A. Caffarelli and P.E. Souganidis.  \newblock A rate of convergence for monotone finite difference approximations to fully nonlinear, uniformly elliptic PDEs. \newblock Comm. Pure Appl. Math., 61 (2008), no. 1, 1Ð17.


\bibitem{CSW}   L. A. Caffarelli, P.E. Souganidis and L.  Wang \newblock Homogenization of fully nonlinear,
uniformly elliptic and parabolic partial differential equations in stationary ergodic
media, \newblock \textit{Comm. Pure Appl. Math.}, 58 (2005), no. 3,  319--361.



\bibitem{ICD}
I.~Capuzzo Dolcetta and H.~Ishii.
\newblock On the rate of convergence in homogenization of {H}amilton-{J}acobi
  equations.
\newblock {\em Indiana Univ. Math. J.}, 50 (2001),  no. 3, 1113--1129.

\bibitem{CaSo}
P.~Cardaliaguet and P.E.~Souganidis.  \newblock  Homogenization and enhancement of the G-equation in random environments. \newblock {\em Comm. Pure and Appl. Math.}, in press.
\newblock


\bibitem{CIL}
M.~G. Crandall, H.~Ishii, and P.-L. Lions.
\newblock User's guide to viscosity solutions of second order partial
  differential equations.
\newblock {\em Bull. Amer. Math. Soc. (N.S.)}, 27 (1992), no. 1, 1--67.

\bibitem{DD} G.~Dal Maso and L.~Modica. \newblock Nonlinear stochastic homogenization. {\it Ann. Mat. Pura
Appl.}, 4 (1986), 347-389.

\bibitem{DD1} G.~Dal Maso and L.~Modica. \newblock Nonlinear stochastic homogenization and ergodic theory.
{\it J. Reine Angew. Math.}, 368 (1986),  28-42.

\bibitem{GZ} X. Guo and O. Zeitouni. \newblock 
    Quenched invariance principle for random walks in balanced random environment. \newblock arXiv:1003.3494.
%

%



\bibitem{Ko}  S.~M.~Kozlov. \newblock The averaging method and walks in inhomogeneous environments.
{\it Uspekhi Mat. Nauk},  40 (1985), no. 2, 61--120.


\bibitem{KRV06}
E.~Kosygina, F.~Rezakhanlou, and S.~R.~S. Varadhan. 
\newblock Stochastic homogenization of {H}amilton-{J}acobi-{B}ellman equations.
\newblock {\em Comm. Pure Appl. Math.}, 59 (2006), no. 10, 1489--1521.


\bibitem{KV}  E.~Kosygina and S. R. S. Varadhan. \newblock
Homogenization of Hamilton-Jacobi-Bellman equations with
respect to time-space shifts in a stationary ergodic medium.
{\it Comm. Pure Appl. Math.},  61 (2008), no. 6, 816-847.


\bibitem{L} G.~Lawler. \newblock Weak convergence of random walk in random environments. \newblock {\em Comm. in Mathematical Phys.}  87 (1982),  81--87.

\bibitem{Lin} J.~Lin. \newblock On the Stochastic homogenization of fully nonlinear uniformly parabolic equations in stationary ergodic 
spatio-temporal media. \newblock arXiv:1307.4743. 



\bibitem{LS2005}
P.-L. Lions and P.~E. Souganidis.
\newblock Homogenization of degenerate second-order PDE in periodic and almost periodic environments and applications.
\newblock {\em Annales de l'Institut Henri Poincare (C) Non Linear Analysis.}, 22 (2005), 667-677.

\bibitem{LS1}
P.-L. Lions and P.~E. Souganidis.
\newblock Correctors for the homogenization of {H}amilton-{J}acobi equations in
  the stationary ergodic setting.
\newblock {\em Comm. Pure Appl. Math.}, 56 (2003), no. 10, 501--1524.

\bibitem{LS2}
P.-L. Lions and P.~E. Souganidis.
\newblock Homogenization of ``viscous'' {H}amilton-{J}acobi equations in
  stationary ergodic media.
\newblock {\em Comm. Partial Differential Equations},  30 (2005), no. 1-3, 335--375.

\bibitem{LS3}
P.-L. Lions and P.~E. Souganidis.
\newblock Stochastic homogenization of {H}amilton-{J}acobi and
  ``viscous''-{H}amilton-{J}acobi equations with convex
  nonlinearities---revisited.
\newblock {\em Commun. Math. Sci.},  8 (2010), no. 2, 627--637.




\bibitem{MN}
I.~Matic and J.~Nolen.
\newblock A sublinear variance bound for solutions of a random
  {H}amilton-{J}acobi equation.
\newblock {\em Journal of Statistical Physics},149 (2012), no 2,  342-361.

\bibitem{NolenNovikov} J.~ Nolen and A.~Novikov.
Homogenization of the G-equation with incompressible random drift.
{\it Commun. Math. Sci.}, 9 (2011), no. 2, 561-582.



\bibitem{Ow} H.~Owhadi. \newblock Approximation of effective conductivity of ergodic media by periodization.
\newblock {\em Probab. Theory Related Fields} 125 (2003), 225--258.

\bibitem{PV1} G.~Papanicolaou and S. R. S.~Varadhan. \newblock Boundary value problems with rapidly oscillating
random coefficients. {\it In Random Fields, Vol. I, II (Esztergom, 1979), volume 27
of Colloq. Math. Soc. Janos Bolyai, 835 -873.} North-Holland, Amsterdam, 1981.

\bibitem{PV2} G.~Papanicolaou and S. R. S.~Varadhan. \newblock Diffusions with random coefficients.
{\it In Statistics and probability: essays in honor of C. R. Rao, 547--552.} North-
Holland, Amsterdam, 1982.


\bibitem{RT}
F.~Rezakhanlou and J.~E. Tarver.
\newblock Homogenization for stochastic {H}amilton-{J}acobi equations.
\newblock {\em Arch. Ration. Mech. Anal.}, 151 (2000), no. 4, 277--309.

\bibitem{S91}
P.~E. Souganidis.
\newblock Stochastic homogenization of {H}amilton-{J}acobi equations and some
  applications.
\newblock {\em Asymptot. Anal.}, 20 (1999), 1--11.

\bibitem{Sch}  R.~Schwab. \newblock Stochastic homogenization of Hamilton-Jacobi equations in stationary ergodic spatio-temporal media.
{\it Indiana Univ. Math. J.} {\bf 58} (2009), no. 2, 537--581.
 

%
%
%

\bibitem{ZKO}  V. V.  Zhikov, S. M. Kozlov and Ole\u{\i}nik, O.   Averaging of parabolic operators.
{\it Trudy Moskov. Mat. Obshch.} 45 (1982), 182--236.
\end{thebibliography}
\end{document}